\documentclass[smallextended]{svjour3}

\usepackage{microtype}
\usepackage{xspace}
\usepackage{url}
\usepackage[utf8]{inputenc}
\usepackage{amsfonts}
\usepackage{amssymb}
\usepackage[cmex10]{amsmath}
\usepackage{algpseudocode}
\usepackage{grf}

\usepackage[utf8]{inputenc}
\usepackage[T1]{fontenc}

\usepackage{url}
\usepackage{doi}
\usepackage[numbers,sort,compress]{natbib}
\usepackage{newfloat}
\usepackage{hyperref}
\hypersetup{
   unicode=true,
   pdftoolbar=false,
   pdfmenubar=false,
   pdffitwindow=false,
   pdfnewwindow=true,
   pdfkeywords={},
   colorlinks=true,
   linkcolor=black,
   citecolor=black,
   filecolor=black,
   urlcolor=black
}

\DeclareFloatingEnvironment[fileext=frm,placement={!ht},name=Algorithm]{algorithm}

\DeclareMathOperator{\st}{s.t.}
\DeclareMathOperator{\sign}{sign}
\DeclareMathOperator*{\argmin}{argmin}

\DeclareMathOperator{\proj}{\mathcal{P}}

\newcommand{\vect}{\boldsymbol}

\title{$\epsilon$-Subgradient Algorithms for Bilevel Convex Optimization\thanks{E.S. Helou was supported by FAPESP grants 2013/07375-0 and 2013/16508-3 and CNPq grant 311476/2014-7. L.E.A. Simões was supported by FAPESP grant 2013/14615-7.}}

\begin{document}

   \author{%
      Elias S. Helou and Lucas E. A. Simões%
   }

   \institute{%
      Elias S. Helou \at Department of Applied Mathematics and Statistics, Institute of Mathematical Sciences and Computation -- USP, São Carlos, \href{mailto:elias@icmc.usp.br}{\tt elias@icmc.usp.br}\\
      Lucas E. A. Simões \at Institute of Mathematics, Statistics and Scientific Computing -- UNICAMP, Campinas
   }%

	\maketitle

	\begin{abstract}
            This paper introduces and studies the convergence properties of a new class of explicit $\epsilon$-subgradient methods for the task of minimizing a convex function over the set of minimizers of another convex minimization problem. The general algorithm specializes to some important cases, such as first-order methods applied to a varying objective function, which have computationally cheap iterations.

            We present numerical experimentation regarding certain applications where the theoretical framework encompasses efficient algorithmic techniques, enabling the use of the resulting methods to solve very large practical problems arising in tomographic image reconstruction.
	\end{abstract}



   \markboth{Helou and Simões}{Subgradient Bilevel Convex Optimization}

	\section{Introduction}

Our aim in the present paper is to solve a \emph{bilevel} or
\emph{hierarchical} optimization problem of the form
\begin{equation}\label{eq:bilevel}
   \begin{split}
      \min & \quad f_1( \vect x )\\
      \st  & \quad \vect x\in \argmin_{\vect y \in X_0} f_0(
\vect y )\text,
   \end{split}
\end{equation}
where $f_i:\mathbb R^n\to\mathbb R$ ($i \in \{ 0, 1 \}$) are
convex functions and $X_0$ is a nonempty closed convex set.

Bilevel problems like~\eqref{eq:bilevel} have already been
considered in the literature. For example, for the case $X_0
= \mathbb R^n$, Cabot~\cite{cab05} suggests the use of the
following algorithm:
\begin{equation}\label{eq:cabot}
   -\frac{\vect x_{k + 1} - \vect x_{k}}{\lambda_k} \in
\partial_{\epsilon_k}( f_0 + \eta_kf_1 )( \vect x_{k + 1}
)\text,
\end{equation}
where $\partial_\epsilon f( \vect x )$ is the $\epsilon$-subdifferential of $f$ at $\vect x$:
\begin{equation*}
   \partial_\epsilon f( \vect x ) := \{ \vect v : f( \vect y ) \geq f( \vect x ) + \vect v^T( \vect y - \vect x ) - \epsilon, \quad \forall \vect y \in \mathbb R^n \},
\end{equation*}
$\eta_k \to 0^+$ and $\lambda_k$ is a nonnegative stepsize. Such iterations are reminiscent of approximate proximal methods, in the sense that~\eqref{eq:cabot} is equivalent to (we follow the notation of~\cite{cab05}):
\begin{equation}\label{eq:cabot'}
   \vect x_{k + 1} \in \epsilon_k\mbox{{}-{}}\argmin_{\vect x
   \in X_0}\left\{\frac1{2\lambda_k}\| \vect x - \vect x_{k}
   \|^2 + f_0( \vect x ) + \eta_kf_1( \vect x
   )\right\}\text.
\end{equation}
While method~\eqref{eq:cabot} is powerful and conceptually simple, its application may be complicated by the implicit formulation. Assuming differentiability, iteration~\eqref{eq:cabot} can also be interpreted as a discretization of the continuous dynamical system
\begin{equation*}
   \dot{\vect x}( t ) + \nabla f_0\bigl( \vect x( t ) \bigr)
   + \eta( t )\nabla f_1\bigl( \vect x( t ) \bigr) = 0.
\end{equation*}
Another way of discretizing this system is to do it
explicitly, that is, to use iterations similar to
\begin{equation}\label{eq:tobac}
   -\frac{\vect x_{k + 1} - \vect x_{k}}{\lambda_k} \in
   \partial_{\epsilon_k} ( f_0 + \eta_kf_1 )( \vect x_{k} )
   \text.
\end{equation}
or in two steps:
\begin{equation}\label{eq:tobacTwoStep}
   \begin{split}
      -\frac{\vect x_{k + 1 / 2} - \vect x_{k}}{\lambda_k} & {}\in \partial_{\epsilon^0_k} f_0( \vect x_{k} )\\
      -\frac{\vect x_{k + 1} - \vect x_{k + 1 / 2}}{\lambda_k\eta_k} & {}\in \partial_{\epsilon^1_k} f_1( \vect x_{k + 1 / 2} ).
   \end{split}
\end{equation}
Among the consequences of the results we will present in this paper, there are sufficient conditions on the sequences $\{\lambda_k\}$, $\{\epsilon_k\}$ and $\{\eta_k\}$ for the convergence of iterations~\eqref{eq:tobacTwoStep} to the solution of problem~\eqref{eq:bilevel}. In fact, convergence of iterations~\eqref{eq:tobac} could also be proven using our abstract results, but would require more restrictive conditions on $\epsilon_k$ and we will keep this topic off the present paper.

Despite the fact that algorithms \eqref{eq:cabot}~and~\eqref{eq:tobacTwoStep} are formally very similar, they differ significantly in both practical and theoretical aspects. At the practical side, implementation of an explicit iteration like~\eqref{eq:tobacTwoStep} requires little more than evaluation of suitable $\epsilon$-subgradients. On the other hand, while strong convexity makes this minimization perhaps be more computationally amenable than approximately optimizing $f_0 + \eta_kf_1$, computing \eqref{eq:cabot'} is still a nontrivial task. Furthermore, even if it allows for some tolerance $\epsilon_k$ in the optimization subproblem, current theory requires $\sum_{k = 0}^\infty \epsilon_k < \infty$ to ensure convergence of~\eqref{eq:cabot}, which means that this tolerance will decrease quickly.

Arguably, an algorithm such as~\eqref{eq:tobacTwoStep} would be among the easiest to implement methodology for solving~\eqref{eq:bilevel} in the general case. For example, Solodov~\cite{sol07,sol08} has also provided algorithms for such bilevel problems: in~\cite{sol07}, where only the differentiable case is considered, the proposed algorithm has the form~\eqref{eq:tobac} with $\epsilon_k \equiv 0$ and $\lambda_k$ selected through a line search with a sufficient decrease criterion based on the value of $f_0 + \eta_kf_1$. While this procedure is still simple to implement on the differentiable case, descent directions are harder to be found in the presence of nondifferentiability and in~\cite{sol08} a bundle technique is used to this end. This approach requires sophisticated quadratic solvers, which complicates implementations. Another, partial and perturbed, descent method was developed by Helou and De Pierro~\cite{hed11}, where only sufficient decrease of $f_0$, assumed to be smooth, is enforced, alleviating the need of a descent direction for $f_1$. But this technique may still  require multiple evaluations of $f_0$, while an iterative process like~\eqref{eq:tobacTwoStep}, differently, does not require any $f_i$ value because it does not rely on descent criteria. Furthermore, it should be remarked that the fact that \eqref{eq:tobacTwoStep} allows for inexactness in the computation has positive impact in algorithmic performance, which we will illustrate through experimental work. An approach currently available in the literature which is similar to a special case of the techniques that we can analyze within our framework can be found in~\cite{brz13}. In this work the results appear to be restricted to monotone methods for a quadratic residual function, but a practical stopping criterion based on the discrepancy principle is given. Another recent work dealing with this kind of problem is~\cite{bes14} where a first order algorithm is proposed and convergence analysis including rates is provided. While we do not provide convergence rates, the theory we present requires less hypothesis on the objective functions and seems to give rise to a wider range of practical algorithms.

%

\subsection{Contributions and Outline of the Paper}

The main contribution of the present paper is as follows.
The thread lead by Cabot and followed by Solodov is based
on the idea of applying classical convex minimization
algorithms to the ever-changing objective function $f_0 +
\eta_kf_1$. It started with the ``tight'' near-minimization
from~\cite{cab05} and evolved to the less stringent
sufficient decrease policy of~\cite{sol07,sol08}. We here
pave this way one step further by showing that the same
principle is applicable to the more anarchic
nonmonotone $\epsilon$-subgradient techniques.

Several unconstrained optimization algorithms have
iterations that can be described as $\epsilon$-subgradient
steps, among which we can mention the incremental
subgradient methods~\cite{neb01}, the aggregated incremental
gradient of~\cite{bhg07} when applied to the
nondiferentiable case~\cite{hed11}, the recent incremental
proximal method~\cite{ber11}, and Polyak's heavy ball
method (let $\tilde\nabla f( \vect x ) \in \partial f(
\vect x )$):
\begin{equation*}
   \vect x_{k + 1} := \vect x_k - \lambda_k \bigl(
   \tilde\nabla f( \vect x_k ) + \alpha( \vect x_k - \vect
   x_{k - 1}) \bigr).
\end{equation*}
The theory we develop will, therefore, cope with all the just mentioned
cases simultaneously. That is, we show that application of
any of these algorithms to the varying objective function
$f_0 + \eta_k f_1$ will converge to the solution of the
bilevel problem~\eqref{eq:bilevel} under assumptions on the
stepsize $\lambda_k$ which are not much different from those
required in the one level case.

We present numerical experimentation showing the effectiveness of the technique when applied to high-resolution micro-tomographic image reconstruction from simulated and from real synchrotron radiation illumination projection data. In this case the amount of data and the number of variables is very large, requiring efficient algorithms with computationally cheap iterations. In this context, important practical contributions are the introduction of certain perturbed, \textsc{fista}-inspired~\cite{bet09} algorithms, resulting in effective methods for problems like~\eqref{eq:bilevel} with Lipschitz-differentiable $f_0$. Furthermore, when solving an instance with a non-differentiable $f_0$ composed as a sum of many convex functions, incremental techniques are very efficient in the first iterations, also resulting in good algorithmic performance. Both the theoretical analysis and the application of these practical algorithms to the bilevel problem~\eqref{eq:bilevel} are new.

Our methods can be seen as perturbations of classical algorithms, in the spirit of the superiorization approach~\cite{gah14}. However, we show more powerful convergence results because we impose some structure on what would otherwise be called a superiorization sequence. We believe that this is a major contribution of the paper because opens the possibility of pursuing bilevel results alongside with superiorization techniques.

   \section{Theoretical Analysis}

%
%

\subsection{Stepsize Considerations}

Recall from the theory of $\epsilon$-subgradient
methods~\cite{col93} for the one level case (that is,
problem~\eqref{eq:bilevel} with $f_1 \equiv 0$)
that convergence of iterations~\eqref{eq:tobac} to a
solution, under mild extra assumptions of subgradient boundedness,
can be ensured with slowly diminishing stepsizes satisfying:
\begin{equation*}
   \sum_{k = 0}^\infty\lambda_k = \infty\quad\text{and}\quad
   \lambda_k \to 0^+\text.
\end{equation*}
The non-summability hypothesis seems necessary. Vanishing stepsizes, however, may have the negative effect of slowing down asymptotic convergence of the algorithm. Therefore, owing to its computationally cheap iteration, methods like~\eqref{eq:tobac} are usually thought to be most competitive when the problem size is very large, or when highly accurate solutions are not required. However, in some important particular (with smooth primary objective function $f_0$) cases supported by the theory developed here, the stepsize $\lambda_k$ does not necessarily have to vanish and in such applications we obtain reasonably fast algorithms. Both stepsize regimes (decreasing and non-decreasing) are evaluated in the experimental work we present and we shall see that incremental techniques are efficient too in the cases where its good characteristics apply, even if with vanishing stepsizes.

For the classical (one level) convex optimization problem,
proximal methods~\cite{roc76} require $\lambda_k \geq
\underline\lambda$ for some $\underline\lambda > 0$. For the
bilevel case, the same stepsize requirement, with an extra 
upper boundedness assumption, i.e.,
\begin{equation}\label{eq:cabotlambda}
    0 < \underline\lambda \leq \lambda_k \leq \overline\lambda,
\end{equation}
would ensure convergence of~\eqref{eq:cabot}
to the optimizer of~\eqref{eq:bilevel}.
For such results to hold~\cite{cab05}, the extra assumption of a
\emph{slow control}:
\begin{equation}\label{eq:caboteta}
    \sum_{k=0}^\infty \eta_k = \infty\quad\text{and}\quad \eta_k \to 0^+,
\end{equation}
was made in order to ensure that the influence of $f_1$ throughout the iterations was strong enough while still becoming arbitrarily small. The alternative form below is more appropriate to us this time, as it will generalize immediately to our algorithms:
\begin{equation}\label{eq:slowcontrol}
   \sum_{k = 0}^\infty \lambda_k\eta_k = \infty
   \quad\text{and}\quad \eta_k \to 0^+\text.
\end{equation}
Notice that if~\eqref{eq:cabotlambda} holds, then~\eqref{eq:slowcontrol} is equivalent  to~\eqref{eq:caboteta}. However, because the net  contribution to each iteration from $\partial f_1( \vect x_k )$ in algorithm~\eqref{eq:tobac} is actually $O( \lambda_k\eta_k )$, \eqref{eq:slowcontrol} generalizes to the case $\lambda_k \to 0^+$, while~\eqref{eq:caboteta} does not.

\subsection{Formal Algorithm Description}

The stepping stone of our analysis will be an abstract
three-step algorithm given as follows:
\begin{equation}\label{eq:algo}
   \begin{split}
      \vect x_{k + 1 / 3} &{}:= \mathcal{O}_{f_0}( \lambda_k, \vect x_k )\text;\\
      \vect x_{k + 2 / 3} &{}:= \mathcal{O}_{f_1}( \mu_k, \vect x_{k + 1 / 3} )\text;\\
      \vect x_{k + 1} &{}:= \mathcal{P}_{X_0}( \vect x_{k + 2 / 3} )\text,
   \end{split}
\end{equation}
where the operators $\mathcal{O}_{f_i}$ for $i \in \{ 0, 1 \}$ have specific conceptual roles and must satisfy certain corresponding properties, which we will discuss right next, and, for a nonempty convex and closed set $X$, $\mathcal{P}_{X}$ is the projector:
\begin{equation*}
   \mathcal{P}_X( \vect x ) := \argmin_{\vect y \in X}\| \vect x - \vect y \|.
\end{equation*}
Sequences $\{ \lambda_k \}$ and $\{ \mu_k \}$ are stepsize sequences. The first one plays, in this abstract setting, the same role it plays in iterative scheme~\eqref{eq:tobac}, while sequence $\{ \mu_k \}$ should be identified with $\{ \lambda_k\eta_k \}$. 
Therefore, application of $\mu_k=\lambda_k\eta_k$ in~\eqref{eq:slowcontrol} leads immediately to
\begin{equation*}\label{eq:muassump}
   \sum_{k = 0}^\infty\mu_k = \infty\quad\text{and}\quad
   \frac{\mu_k}{\lambda_k} \to 0^+\text.
\end{equation*}

Let us then describe the properties required for the \emph{optimality operators} $\mathcal O_{f_i}$. The imposed characteristics are easy to meet, as we later illustrate.

\begin{property}\label{prop:propopt1}
   There is $\beta > 0$ such that for any $\lambda \geq 0$ and for all $\vect x_{k + i / 3}$, $\vect y \in \mathbb R^n$, and $i \in \{ 0, 1 \}$:
   \begin{multline*}
      \| \mathcal O_{f_i}( \lambda, \vect x_{k + i / 3} ) - \vect y \|^2 \leq \| \vect x_{k + i / 3} - \vect y \|^2 - \beta\lambda\bigl( f_i\bigl( \mathcal O_{f_i}( \lambda, \vect x_{k + i / 3} ) \bigr) - f_i( \vect y ) \bigr)\\
      {} + \lambda\rho_i( \lambda, k )\text,
   \end{multline*}
   where $\rho_i( \lambda, k )$ represents an error term, with properties to be describe later.
\end{property}

Below the description of the next property, we give an example of a class of operators which satisfy this condition. Furthermore, Subsections~\ref{subsec:algolip}, \ref{subsec:algoinc}, and \ref{subsec:secobj} bring four other instances that will be used in the experimental part of the paper: the projected gradient, the incremental subgradient, the proximal map and the iterated subgradient step. In fact, the key utility of this abstract definition is to be able to encompass several useful classical optimization steps while still ensuring sufficient qualities in order to provide convergence results. For this to be true, the error term will have to be controlled in a specific way, but, for every case we have found, the error term magnitude is bounded by a constant times the stepsize and this way we can always obtain convergent algorithms by selecting proper stepsize sequences.

%
\begin{property}\label{prop:propoptbound}
   There exists $\gamma > 0$ such that
   \begin{equation*}
      \| \vect x - \mathcal{O}_{f_i}( \lambda, \vect x )
      \|_2 \leq \lambda\gamma\text.
   \end{equation*}
\end{property}

Property~\ref{prop:propopt1} guarantees that, going from some fixed $\vect x$, the operator $\mathcal O_{f_i}$ will approach a point $\vect y$ with improved $f_i$ value if only the result of the operation does not have a better $f_i$ value than $\vect y$ and the stepsize $\lambda$ is small enough.
Property~\ref{prop:propoptbound} is no
more than a boundedness assumption on the operators which makes sure that the stepsize controls the magnitude of the movement.

These can be derived from somewhat standard hypothesis for $\epsilon$-subgradient
algorithms (see, e.g., \cite{col93}) and, as such, a
plethora of concrete realizations of such operators $\mathcal O_f$ is
possible, the most obvious being $\epsilon$-subgradient steps:
\begin{equation*}
   \mathcal S_f( \lambda, \vect x ) := \vect x - \lambda\tilde\nabla_{\epsilon} f( \vect x )\text,
\end{equation*}
where $\tilde\nabla_\epsilon f( \vect x ) \in
\partial_\epsilon f( \vect x )$. In this case we have:
\begin{equation}\label{eq:subgradstepineq}
   \|\mathcal S_f( \lambda, \vect x ) - \vect y\|_2^2 \leq
   \| \vect x - \vect y \|_2^2 - 2\lambda\bigl( f( \vect x )
   - f( \vect y ) \bigr) + \lambda\bigl( \lambda \|
   \tilde\nabla_{\epsilon}f( \vect x ) \|_2^2
   + 2\epsilon \bigr)\text.
\end{equation}
Denote $\vect z = \vect x - \lambda \tilde\nabla_{\epsilon} f( \vect x )$. Then, convexity leads to
\begin{equation*}
   f( \vect x ) \geq f( \vect z ) + \tilde\nabla f( \vect z )^T( \vect x - \vect z ) = f( \vect z ) + \lambda\tilde\nabla f( \vect z )^T\tilde\nabla_{\epsilon} f( \vect x ),
\end{equation*}
where $\tilde\nabla f( \vect z ) \in \partial f( \vect z )$. Multiplying the above inequality by $-2\lambda$ we get
\begin{equation*}
   -2\lambda f( \vect x ) \leq -2\lambda \Bigl( f\bigl( \mathcal S_f( \lambda, \vect x ) \bigr) + \lambda \tilde\nabla f\bigl( \mathcal S_f( \lambda, \vect x ) \bigr)^T\tilde\nabla_{\epsilon} f( \vect x ) \Bigr)
\end{equation*}
This with~\eqref{eq:subgradstepineq} gives
\begin{multline}\label{eq:subgradstepineqB}
   \|\mathcal S_f( \lambda, \vect x ) - \vect y\|_2^2 \leq
   \| \vect x - \vect y \|_2^2 - 2\lambda\bigl( f( \mathcal S_f( \lambda, \vect x ) )
   - f( \vect y ) \bigr) \\
   {} + \lambda\bigl( \lambda \| \tilde\nabla_{\epsilon}f( \vect x ) \|_2^2 + 2\epsilon - 2\lambda \tilde\nabla f\bigl( \mathcal S_f( \lambda, \vect x ) \bigr)^T\tilde\nabla_{\epsilon} f( \vect x ) \bigr)\text,
\end{multline}
so that we can satisfy Properties~\ref{prop:propopt1}~and~\ref{prop:propoptbound} for $\mathcal O_f = \mathcal S_f$ if we further assume $\epsilon$-subgradient boundedness (and consequently subgradient boundedness).

A straightforward generalization of the argument leading from~\eqref{eq:subgradstepineq} to~\eqref{eq:subgradstepineqB}, omitted for brevity, results in the following statement, which will be useful later:
\begin{proposition}\label{prop:oneIsOther}
   Assume an operator $\mathcal O_f : \mathbb R \times \mathbb R^n \to \mathbb R^n$ satisfies, for $\lambda > 0$
   \begin{equation*}
      \|\mathcal O_f( \lambda, \vect x ) - \vect y\|_2^2 \leq \| \vect x - \vect y \|_2^2 - 2\lambda\bigl( f( \vect x ) - f( \vect y ) \bigr) + \lambda\varrho( \lambda )\text,
   \end{equation*}
   where $\varrho( \lambda )$ is an error term, and
   \begin{equation*}
      \|\mathcal O_f( \lambda, \vect x ) - \vect x\| \leq \lambda\gamma,
   \end{equation*}
   for some $\gamma > 0$. Then, we have:
   \begin{multline*}
      \|\mathcal O_f( \lambda, \vect x ) - \vect y\|_2^2 \leq \| \vect x - \vect y \|_2^2 - 2\lambda\bigl( f( O_f( \lambda, \vect x ) ) - f( \vect y ) \bigr)\\
      {}+ \lambda\Bigl( \varrho( \lambda ) + 2 \lambda\gamma\bigl\| \tilde\nabla f\bigl( \mathcal O_f( \lambda, \vect x ) \bigr) \bigr\| \Bigr)\text,
   \end{multline*}
   where $\tilde\nabla f\bigl( \mathcal O_f( \lambda, \vect x ) \bigr) \in \partial f\bigl( \mathcal O_f( \lambda, \vect x ) \bigr)$.
\end{proposition}

Our analysis will focus on algorithms more general than~\eqref{eq:tobac}, allowing simple constraint sets $X_0$ to be handled. We recall that the projection onto a nonempty convex closed set $X_0$ satisfies:
\begin{property}
   For all $\vect x \in \mathbb R^n$ and $\vect y \in X_0$, we have
   \begin{equation}\label{eq:propfacforte1}
      \| \proj_{X_0}( \vect x ) - \vect y \| \leq \| \vect x - \vect y \|\text.
   \end{equation}
\end{property}



\subsection{Convergence Results}

We introduce some simplifying notations:
\begin{itemize}
   \item $f_i^*$ is the optimal value of $f_i$ over $X_i$ for $i \in \{ 0, 1 \}$ where;
   \item $X_0$ is given and $X_{i + 1} := \{ \vect x \in X_i : f_i( \vect x )
   = f_i^*\}$ for $i \in \{ 0, 1 \}$;
   \item $[ x ]_+ := \max\{ 0, x \}$;
   \item $d_X( \vect x ) := \| \vect x - \mathcal P_X( \vect x ) \|$.
\end{itemize}

Our first result shows convergence of the iterates to the set of minimizers of $f_0$ over $X_0$. We next prove convergence to the set of minimizers of $f_1$ over $X_1$. Both of these preliminary results contain certain technical and some apparently strong hypothesis. We subsequently weaken and clarify such \emph{ad hoc} requirements in order to obtain our main results.

\begin{proposition}\label{prop:CX}
   Assume that $X_1 \neq \emptyset$, $X_1$ is bounded (or $\{ \vect x_{k} \}$ is bounded) $\sum_{i = 0}^\infty \lambda_k = \infty$, $\mathcal O_{f_0}$ and $\mathcal O_{f_1}$ satisfy Property~\ref{prop:propopt1}, $\mathcal O_{f_1}$ satisfy also Property~\ref{prop:propoptbound}, $f_1( \vect x_{k + 2 / 3} ) \geq \underline f > -\infty$, $\| \vect x_{k} - \vect x_{k + 1 / 3} \| \to 0$, $\rho_0( \lambda_k, k ) \to 0$, $\rho_1( \mu_k, k ) \leq \overline\rho_1 < \infty$, $\mu_k \to 0$, and $\mu_k / \lambda_k \to 0$. Suppose also that there exists $M$ such that $\forall k \in \mathbb N$ there is $\vect v_k \in \partial f_0( \vect x_{k} )$ for which $\| \vect v_k \| < M$, then we have
   \begin{equation*}
      \lim_{k\to\infty}d_{X_1}( \vect x_k ) = 0\text.
   \end{equation*}
\end{proposition}

\begin{proof}
   First notice that $\mu_k \to 0$ and Property~\ref{prop:propoptbound} imply $\| \vect x_{k + 1 / 3} - \vect x_{k + 2 / 3} \| \to 0$. Then we take into consideration the non-expansiveness of the projection and of  Property~\ref{prop:propopt1} of $\mathcal O_{f_0}$ and $\mathcal O_{f_1}$, there holds, for $\vect y \in X_1$:
   \begin{equation}\label{eq:lemmCXaux1}
      \begin{split}
         \| \vect x_{k + 1} - \vect y \|^2 & {} \leq \| \vect x_{k + 2 / 3} - \vect y \|^2 \\
         &{}\leq \| \vect x_{k + 1 / 3} - \vect y \|^2 - \beta\mu_k\bigl( f_1( \vect x_{k + 2 / 3} ) - f_1( \vect y ) \bigr) + \mu_k\rho_1( \mu_k, k ) \\
         &{} \leq \| \vect x_{k} - \vect y \|^2 - \beta\lambda_k\bigl( f_0( \vect x_{k + 1 / 3} ) - f_0^* \bigr) + \lambda_k\rho_0( \lambda_k, k ) \\
         & \qquad\qquad\qquad\qquad {} - \beta\mu_k\bigl( f_1( \vect x_{k + 2 / 3} ) - f_1( \vect y ) \bigr) + \mu_k\rho_1( \mu_k, k )\text.
      \end{split}
   \end{equation}
   Denote $N = f_1( \vect y ) - \underline f$, thereby simplifying the above expression to:
   \begin{multline}\label{eq:approx_f0_inter}
      \| \vect x_{k + 1} - \vect y \|^2 \leq \| \vect x_{k} - \vect y \|^2 - \beta\lambda_k\bigl( f_0( \vect x_{k + 1 / 3} ) - f_0^* \bigr) \\
      {} + \lambda_k\rho_0( \lambda_k, k ) + \mu_k\bigl( \beta N + \rho_1( \mu_k, k ) \bigr)\text.
   \end{multline}
   Then, the boundedness of $\partial f_0( \vect x_{k} )$ leads to
   \begin{equation*}
      f_0( \vect x_{k + 1 / 3} ) \geq f( \vect x_{k} ) - M\| \vect x_{k + 1 / 3} - \vect x_{k} \|\text,
   \end{equation*}
   which together with~\eqref{eq:approx_f0_inter} gives
   \begin{multline}\label{eq:approx_f0}
      \| \vect x_{k + 1} - \vect y \|^2 \leq \| \vect x_{k} - \vect y \|^2 - \beta\lambda_k\bigl( f_0( \vect x_{k} ) - f_0^* \bigr) \\
      {} + \lambda_k\bigl( \rho_0( \lambda_k, k ) + \beta M\| \vect x_{k + 1 / 3} - \vect x_{k} \| \bigr) + \mu_k\bigl( \beta N + \rho_1( \mu_k, k ) \bigr)\text.
   \end{multline}

   We shall denote, for $\delta \geq 0$:
   \begin{equation*}
      X_1^\delta := \{ \vect x_k : f_0( \vect x_k ) \leq f_0^* + \delta \}.
   \end{equation*}
   Notice that if $X_1$ is bounded (or $\{ \vect x_{k} \}$ is bounded), then $X_1^\delta$ is bounded. Therefore, the following quantity is well defined:
   \begin{equation*}
      \Delta_1( \delta ) := \sup_{\vect x \in X_1^\delta} d_{X_1}( \vect x )\text.
   \end{equation*}
   Furthermore, we have
   \begin{equation*}
      \lim_{\eta \to 0} \Delta_1( \delta + \eta ) = \Delta_1( \delta )\quad\text{and}\quad \Delta_1( 0 ) = 0\text.
   \end{equation*}

   Let $\delta$ be any positive real number and consider, with $\rho_1( \mu_k, k ) \leq \overline\rho_1$, $\mu_k/\lambda_k \to 0$, $\| \vect x_{k + 1 / 3} - \vect x_{k} \| \to 0$, and $\rho_0( \lambda_k, k ) \to 0$ in mind, that $k_0$ is large enough such that $k \geq k_0$ implies
   \begin{equation}\label{eq:smallness_f0}
      \rho_0( \lambda_k, k ) + \beta M\| \vect x_{k + 1 / 3} - \vect x_{k} \| < \beta\frac{\delta}{3}, \quad\text{and}\quad
      \frac{\mu_k}{\lambda_k}\bigl( \beta N + \rho_1( \mu_k, k ) \bigr) < \beta\frac{\delta}{3}\text.
   \end{equation}
   Then, two situations may occur:
   \begin{enumerate}
      \item $d_{X_1}( \vect x_k ) \geq \Delta_1( \delta )$;\label{item:lemmCXcase1}
      \item $d_{X_1}( \vect x_k ) <    \Delta_1( \delta )$.
   \end{enumerate}

   Let us first suppose that Case~\ref{item:lemmCXcase1} holds, that is $f_0( \vect x_{k} ) - f_0^* \geq \delta$. Then, for $k \geq k_0$, from~\eqref{eq:approx_f0}~and~\eqref{eq:smallness_f0} we get:
   \begin{equation*}
      \| \vect x_{k + 1} - \vect y \|^2 \leq \| \vect x_{k} - \vect y \|^2 - \beta\lambda_k\frac\delta3.
   \end{equation*}
   In particular,
   \begin{equation*}
      d_{X_1}( \vect x_{k + 1} )^2 \leq \| \vect x_{k + 1} - \proj_{X_1}( \vect x_k ) \|^2 \leq d_{X_1}( \vect x_k )^2 - \beta\lambda_k\frac\delta3\text.
   \end{equation*}
   Therefore, since $\sum_{k = 0}^\infty\lambda_k = \infty$, there must exist an arbitrarily large $k_1 \geq k_0$ such that $d_{X_1}( \vect x_k ) < \Delta_1( \delta )$.

   Now, let us notice that, because of~\eqref{eq:propfacforte1}
   \begin{equation}\label{eq:nextdistbound}
      \begin{split}
         d_{X_1}( \vect x_{k + 1} ) &{}\leq \| \vect x_{k + 1} - \proj_{X_1}( \vect x_{k} ) \| \\
         & {}\leq \| \vect x_{k + 2 / 3} - \proj_{X_1}( \vect x_{k} ) \| \\
         & {}\leq d_{X_1}( \vect x_k ) + \| \vect x_{k} - \vect x_{k + 2 / 3} \|\text.
      \end{split}
   \end{equation}
   Given the hypothesis, we may assume that $k_0$ is large enough such that, in addition to~\eqref{eq:smallness_f0}, we have also
   \begin{equation*}
      \| \vect x_{k} - \vect x_{k + 2 / 3} \| \leq \delta.
   \end{equation*}
   Therefore, for $k > k_1$, there holds:
   \begin{equation*}
      d_{X_1}( \vect x_{k} ) \leq \Delta_1( \delta ) + \delta.
   \end{equation*}
   Since $\delta > 0$ was arbitrary and $\lim_{\delta \to 0}\Delta_1( \delta ) = 0$, the claim is proven.\qed
\end{proof}

%
%
%


\begin{proposition}\label{prop:CX2}
   Assume $X_2 \neq \emptyset$, $X_2$ is bounded (or $\{ \vect x_{k} \}$ is bounded), that $\mu_k \to 0$, $\sum_{i = 0}^\infty \mu_k = \infty$, $\mathcal O_{f_0}$ and $\mathcal O_{f_1}$ satisfy Property~\ref{prop:propopt1}, $\mathcal O_{f_1}$ also satisfies Property~\ref{prop:propoptbound}, $\lambda_k [ f_0^* - f_0( \vect x_{k + 1 / 3} ) ]_+ / \mu_k \to 0$, $d_{X_1}( \vect x_k ) \to 0$, $\| \vect x_{k} - \vect x_{k + 1 / 3} \| \to 0$, $\lambda_k\rho_0( \lambda_k, k ) / \mu_k \to 0$ and $\rho_1( \mu_k, k ) \to 0$. Suppose also that there exists an $M$ such that $\forall k \in \mathbb N$ there are $\vect v_k \in \partial f_0( \vect x_k )$ and $\vect w_k \in \partial f_1\bigl( \proj_{X_0}( \vect x_k ) \bigr)$ for which $\| \vect v_k \| < M$ and $\| \vect w_k \| \leq M$, then we have
   \begin{equation*}
      \lim_{k\to\infty}d_{X_2}( \vect x_k ) = 0\text.
   \end{equation*}
\end{proposition}

\begin{proof}
   Notice for later reference that just like in Proposition~\ref{prop:CX}, the hypotheses imply that $\| \vect x_{k + 1 / 3} - \vect x_{k + 2 / 3} \| \to 0$. Now, if we use~\eqref{eq:lemmCXaux1} with $\vect y \in X_2 \subset X_1$, we get:
   \begin{multline}\label{eq:diffestX1X2}
      \| \vect x_{k + 1} - \vect y \|^2 \leq \| \vect x_{k} - \vect y \|^2 - \beta\mu_k\bigl( f_1( \vect x_{k + 2 / 3} ) - f_1( \vect y ) \bigr) \\
      {} + \lambda_k\rho_0( \lambda_k, k ) + \mu_k\rho_1( \mu_k, k ) + \beta\lambda_k[ f_0^* - f( \vect x_{k + 1 / 3} ) ]_+\text.
   \end{multline}
   Now let $\tilde \nabla f_1( \vect x_k ) \in \partial f_1( \vect x_k )$ and then notice that convexity of $f_1$, Cauchy-Schwarz inequality and the boundedness assumption on $\partial f_1( \vect x_k )$ lead to:
   \begin{equation}\label{eq:f0bound}
      \begin{split}
         f_1( \vect x_{k + 2 / 3} ) &{} \geq f_1( \vect x_{k} ) + \tilde \nabla f_1( \vect x_k )^T( \vect x_{k + 2 / 3} - \vect x_k )\\
                            &{} \geq f_1( \vect x_{k} ) - \|\tilde \nabla f_1( \vect x_k )\|\| \vect x_{k + 2 / 3} - \vect x_k \|\\
                            &{} \geq f_1( \vect x_{k} ) - M\| \vect x_{k + 2 / 3} - \vect x_k \|.
      \end{split}
   \end{equation}
   Then, using~\eqref{eq:f0bound} in~\eqref{eq:diffestX1X2} it is possible to obtain:
   \begin{multline}\label{eq:diffestf1Final}
      \| \vect x_{k + 1} - \vect y \|^2 \leq \| \vect x_{k} - \vect y \|^2 - \beta\mu_k\bigl( f_1( \vect x_{k} ) - f_1^* \bigr) + \lambda_k\rho_0( \lambda_k, k )\\
      {}+ \mu_k\rho_1( \mu_k, k ) + \beta\lambda_k[ f_0^* - f_0( \vect x_{k + 1 / 3} ) ]_+ + \beta\mu_kM\| \vect x_{k + 2 / 3} - \vect x_{k} \|\text.
   \end{multline}

   Similarly to the $\Delta_1$ notation introduced above, we will denote, for $\delta \geq 0$:
   \begin{equation*}
      X_2^\delta := \{ \vect x_k : f_1\bigl( \proj_{X_1}( \vect x_{k} ) \bigr) \leq f_1^* + \delta \}.
   \end{equation*}
   Notice that if $X_2$ is bounded (or $\{ \vect x_{k} \}$ is bounded) and $d_{X_1}( \vect x_{k} )$ is also bounded, then $X_2^\delta$ is bounded. Therefore, the following quantity is well defined:
   \begin{equation*}
      \Delta_2( \delta ) := \sup_{\vect x \in X_2^\delta} d_{X_2}( \vect x )\text.
   \end{equation*}
   Furthermore, we have
   \begin{equation*}
      \lim_{\eta \to 0} \Delta_2( \delta + \eta ) = \Delta_2( \delta )\quad\text{and}\quad \Delta_2( 0 ) = 0\text.
   \end{equation*}

   Given the hypothesis, for any fixed $\delta > 0$, there is $k_0$ such that $k \geq k_0$ implies that
   \begin{multline}\label{eq:smallness_f1}
      \frac{\lambda_k\rho_0( \lambda_k, k )}{\mu_k} < \beta\frac{\delta}{5}, \quad
      \rho_1( \mu_k, k ) < \beta\frac{\delta}{5}, \quad
      \frac{\lambda_k[ f_0^* - f( \vect x_{k + 1 / 3} ) ]_+}{\mu_k} < \frac{\delta}{5},\\
      \text{and} \quad \mu_kM\| \vect x_{k + 2 / 3} - \vect x_{k} \| < \frac{\delta}{5}\text.
   \end{multline}

   We from now on assume $k > k_0$ and split in two different possibilities:
   \begin{enumerate}
      \item $f_1( \vect x_k ) >    f_1^* + \delta$;\label{case:lemmX2_1}
      \item $f_1( \vect x_k ) \leq f_1^* + \delta$.\label{case:lemmX2_2}
   \end{enumerate}

   We start by analyzing Case~\ref{case:lemmX2_1}. Using~\eqref{eq:smallness_f1} in~\eqref{eq:diffestf1Final} we get, for $\vect y \in X_2$:
   \begin{equation*}
      \| \vect x_{k + 1} - \vect y \|^2 < \| \vect x_{k} - \vect y \|^2 - \beta\mu_k\frac\delta5.
   \end{equation*}
   In particular:
   \begin{equation*}
      d_{X_2}( \vect x_{k + 1} )^2 \leq \| \vect x_{k + 1} - \proj_{X_2}( \vect x_k ) \|^2 < d_{X_2}( \vect x_{k} )^2 - \beta\mu_k\frac\delta5.
   \end{equation*}
   Because of $\sum_{k = 0}^\infty \mu_k = \infty$, this inequality means that there is an arbitrarily large $k_1$ such that $f_1( \vect x_{k_1} ) \leq f_1^* + \delta$.

   Let us then focus on Case~\eqref{case:lemmX2_2}. We first notice that the assumed boundedness of $\partial f_1\bigl( \proj_{X_1}( \vect x_k ) \bigr)$ leads to
   \begin{equation*}
      f_1\bigl( \proj_{X_1}( \vect x_k ) \bigr) \leq f_1( \vect x_k ) + Md_{X_1}( \vect x_k ).
   \end{equation*}
   Therefore, $f_1( \vect x_k ) \leq f_1^* + \delta$ implies
   \begin{equation*}
      \vect x_{k} \in X_2^{\delta + Md_{X_1}( \vect x_k )}.
   \end{equation*}
   Thus,~\eqref{eq:nextdistbound} now reads
   \begin{equation*}
      d_{X_2}( \vect x_{k + 1} ) \leq \Delta_2\bigl( \delta + Md_{X_1}( \vect x_k ) \bigr) + \| \vect x_{k + 2 / 3} - \vect x_{k} \|.
   \end{equation*}

   Then, because we have assumed $d_{X_1}( \vect x_k ) \to 0$, $\| \vect x_{k} - \vect x_{k + 1 / 3} \| \to 0$, and $\| \vect x_{k + 1 / 3} - \vect x_{k + 2 / 3} \| \to 0$, we can recall $\lim_{\eta \to 0}\Delta_2( \delta + \eta ) = \Delta_2( \delta )$, so that the argumentation above leads to the conclusion that
   \begin{equation*}
      \limsup_{k \to \infty}d_{X_2}( \vect x_k ) \leq \Delta_2( \delta ).
   \end{equation*}
   Finally, because $\delta > 0$ was arbitrary and $\lim_{\delta \to 0} \Delta_2( \delta ) = 0$, we have just proven the claimed result.\qed
\end{proof}

We now present two different algorithms and prove their convergence based on the above general results. Next section contains numerical experimentation regarding some of these methods in four different bilevel models arising in high-resolution micro-tomographic image reconstruction from synchrotron illumination.

\subsection{An Algorithm for Lipschitz-Differentiable Primary Objective Functions}\label{subsec:algolip}

In this Subsection we suppose $f_0$ in the bilevel optimization problem~\eqref{eq:bilevel} is differentiable with uniformly bounded and Lipschitz continuous gradient. For this kind of problem, we will consider Algorithm~\ref{algo:fiba}, which we name Fast Iterative Bilevel Algorithm (\textsc{fiba}). \textsc{fiba} first performs a projected gradient descent step, with a stepsize that does not change unless the magnitude of this operation is larger than a control sequence. This computation is then followed by the application of an optimality operator of the kind described by Properties~\ref{prop:propopt1}~and~\ref{prop:propoptbound}. Such optimality operator is actually applied to a perturbation of the point obtained by the projected gradient descent, in a fashion similar to the Fast Iterative Soft-Thresholding Algorithm (\textsc{fista})~\cite{bet09}, but with the magnitude of the perturbation bounded by $\mu_k\zeta_k$, where $\{ \zeta_k \}$ is a positive vanishing sequence.

\begin{algorithm}
   \begin{algorithmic}[1]
      \Require{$\vect x_{0}$, $\{\lambda_k\}$, $\{\mu_k\}$, $\{\zeta_k\}$}%
      \Statex%
      \State{Initialization: $k \leftarrow 0$, $t_0 = 1$, $\vect x_{-2 / 3} = \vect x_{0}$, $i_0 = 0$ }%
      \Statex%
      \Repeat
         \Statex%
         \State{$\vect x_{k + 1 / 3} = \proj_{X_0}\bigl( \vect x_k - \lambda_{i_k}\nabla f_0( \vect x_k ) \bigr)$}\label{line:bifaStep1}
         \If{$\| \vect x_k - \vect x_{k + 1 / 3} \| \geq \zeta_k$}\label{line:bifaStep2}
            \State{$i_{k + 1} = i_{k} + 1$}%
         \Else%
            \State{$i_{k + 1} = i_{k}$}%
         \EndIf\label{line:bifaStep3}
         \Statex%
         \State{$t_{k + 1} = \frac{1 + \sqrt{ 1 + 4t_{k}^2 }}{2}$, $\xi_k = \min\left\{ 1, \frac{\mu_k\zeta_k}{\| \vect x_{k + 1 / 3} - \vect x_{( k - 1 ) + 1 / 3} \|} \right\}$}%
         \State{$\vect y_{k + 1 / 3} = \vect x_{k + 1 / 3} + \xi_k\left( \frac{t_{k} - 1}{t_{k + 1}} \right)( \vect x_{k + 1 / 3} - \vect x_{( k - 1 ) + 1 / 3} )$}%
         \Statex%
         \State{$\vect x_{k + 2 / 3} = \mathcal O_{f_1}( \vect y_{k + 1 / 3}, \mu_k )$}%
         \State{$\vect x_{k + 1} = \proj_{X_0}( \vect x_{k + 2/ 3} )$}%
         \Statex%
         \State{$k \leftarrow k + 1$}%
         \Statex%
      \Until{convergence is reached}
   \end{algorithmic}
   \caption{Fast Iterative Bilevel Algorithm}\label{algo:fiba}
\end{algorithm}

In order to analyze convergence of Algorithm~\ref{algo:fiba} through our previous results, we first look at the simple projected gradient descent
\begin{equation*}
   \mathcal G_f( \lambda, \vect x ) := \proj_{X_0}\bigl( \vect x - \lambda \nabla f( \vect x ) \bigr)
\end{equation*}
as an instance of the optimality operators considered above. Let $L_f$ denote the Lipischtz constant of $\nabla f$. Then, if $\lambda \leq 1 / L_f$, it is possible to show (see, e.g., \cite{bet09} and references therein) that
\begin{equation}\label{eq:upperBoundLips}
   f( \vect y ) \leq f( \vect x ) + \nabla f( \vect x )^T( \vect y - \vect x ) + \frac 1{2\lambda}\|\vect y - \vect x \|^2.
\end{equation}
Let now $\iota_{X_0}$ be the indicator function:
\begin{equation}\label{eq:iota}
   \iota_{X_0}( \vect x ) :=
      \begin{cases}
         \infty & \text{if} \quad \vect x \notin X_0\\
         0      & \text{if} \quad \vect x \in X_0\text.
      \end{cases}
\end{equation}
Then, inequality~\eqref{eq:upperBoundLips} leads to
\begin{equation*}
   f( \vect y ) + \iota_{X_0}( \vect y ) \leq f( \vect x ) + \nabla f( \vect x )^T( \vect y - \vect x ) + \frac 1{2\lambda}\|\vect y - \vect x \|^2 + \iota_{X_0}( \vect y ).
\end{equation*}
Therefore,~\cite[Lemma~2.3]{bet09} can be used with $L = 1 / \lambda$, $g = \iota_{X_0}$, $\mathbf y = \vect x$ and $\mathbf x = \vect y$ in order to get, for $\vect y \in X_0$:
\begin{equation}\label{eq:fibaopt}
   \begin{split}
      2\lambda\Bigl( f( \vect y ) - f\bigl( \mathcal G_f( \lambda, \vect x ) \bigl) \Bigr) & {}\geq \| \vect x - \mathcal G_f( \lambda, \vect x ) \|^2 + 2( \vect x - \vect y )^T\bigl( \mathcal G_f( \lambda, \vect x ) - \vect x \bigr) \\
      & {} = \| \vect y - \mathcal G_f( \lambda, \vect x ) \|^2 - \| \vect y - \vect x \|^2.
   \end{split}
\end{equation}
Thus, for $\lambda \leq 1/L_f$, we can see that $\mathcal G_f( \lambda, \vect x )$ satisfies Property~\ref{prop:propopt1} with $\beta = 2$ and $\rho( \lambda, k ) \equiv 0$. Notice that the operation described at line~\ref{line:bifaStep1} of Algorithm~\ref{algo:fiba} is actually:
\begin{equation*}
   \vect x_{k + 1 / 3} = \mathcal G_{f_0}( \lambda_{i_k}, \vect x_k ).
\end{equation*}
Consequently, according to~\eqref{eq:fibaopt}, Algorithm~\ref{algo:fiba} is an instance of~\eqref{eq:algo} with an optimality operator $\mathcal O_{f_0}$ which satisfies Property~\ref{prop:propopt1} with $\rho_0( \lambda, k ) \equiv 0$, whenever $\lambda \leq 1 / L_f$ or if, which is weaker, \eqref{eq:upperBoundLips} holds with $\vect x = \vect x_{k}$ and $\vect y = \vect x_{k + 1 / 3}$. This characteristic of the error term implies that it is possible to have a precise enough operator without requiring $\lambda_{i_k} \to 0$. However, $\| \vect x_k - \vect x_{ k + 1 / 3 } \| \to 0$ would still require $\lambda_{i_k} \to 0$ if the projection on line~\ref{line:bifaStep1} of Algorithm~\ref{algo:fiba} were not performed and this is the reason why there are two projections in this method.

Now, let us recall that $\vect x_{k} \in X_0$, so that by the definition of $\mathcal G_f$ and~\eqref{eq:propfacforte1}
\begin{equation}\label{eq:gradBoundDiff}
   \| \vect x_{k} - \mathcal G_f( \lambda_{i_k}, \vect x_k ) \| \leq \lambda_{i_k}\| \nabla f( \vect x_k ) \|.
\end{equation}
Therefore, if the sequence $\{ \nabla f_0( \vect x_{k} ) \}$ is bounded, and if $\lambda_k \to 0$ and $\zeta_k \to 0$, then the procedure in lines~\ref{line:bifaStep2}--\ref{line:bifaStep3} of Algorithm~\ref{algo:fiba} implies that $\| \vect x_{k + 1 / 3} - \vect x_{k} \| \to 0$. Observe also that if $\sum_k\lambda_k = \infty$, then $\sum_k\lambda_{i_k} = \infty$ too.

Now, we consider the fact that the optimization operator for the secondary function $f_1$ is used in a perturbed point $\vect y_{k + 1 / 3}$, instead of at $\vect x_{k + 1 / 3}$. Our goal is to verify that the relevant properties of $\mathcal O_{f_1}$ are maintained. The first observation is that, given the way that $\vect y_{k + 1 / 3}$ is defined, we have
\begin{equation}\label{eq:boundmuzeta}
   \| \vect x_{k + 1 / 3} - \vect y_{k + 1 / 3} \| \leq \mu_k\zeta_k.
\end{equation}
We then define a new operator $\tilde{\mathcal O}_{f_1}$, based on $\mathcal O_{f_1}$, as follows:
\begin{equation*}
   \tilde{\mathcal O}_{f_1}( \mu, \vect x_{k + 1 / 3} ) := \mathcal O_{f_1}( \mu, \vect y_{k + 1 / 3} ).
\end{equation*}
Notice that $\tilde{\mathcal O}_{f_1}$ role is to hide the perturbation from the analysis. Also, this kind of operator is the reason why we use an iteration-dependent error term in Property~\ref{prop:propopt1}, as we will see just below.
Let us then assume that Property~\ref{prop:propopt1} holds for $O_{f_1}$, therefore:
\begin{equation}\label{eq:boundtildeo}
   \begin{split}
      \| \vect y - \tilde{\mathcal O}_{f_1}( \mu_k, \vect x_{k + 1 / 3} ) \|^2 & {}= \| \vect y - \mathcal O_{f_1}( \mu_k, \vect y_{k + 1 / 3} ) \|^2\\
      &{} \leq \| \vect y - \vect y_{k + 1 / 3} \|^2 - 2\mu_k\bigl( f( \mathcal O_{f_1}( \mu_k, \vect y_{k + 1 / 3} ) ) - f( \vect y ) \bigr)\\
      & \qquad\qquad\qquad\qquad\qquad\qquad\qquad\qquad\quad {}+ \mu_k\rho_1( \mu_k, k )\\
      &{} = \| \vect y - \vect y_{k + 1 / 3} \|^2 - 2\mu_k\bigl( f( \tilde{\mathcal O}_{f_1}( \mu_k, \vect x_{k + 1 / 3} ) ) - f( \vect y ) \bigr)\\
      & \qquad\qquad\qquad\qquad\qquad\qquad\qquad\qquad\quad {}+ \mu_k\rho_1( \mu_k, k ).\\
   \end{split}
\end{equation}
Now, by taking~\eqref{eq:boundmuzeta} into consideration, a straightforward computation leads to
\begin{equation*}
   \| \vect y - \vect y_{k + 1 / 3} \|^2 \leq \| \vect y - \vect x_{k + 1 / 3} \|^2 + \mu_k\zeta_k\left( 2\| \vect y -\vect x_{k + 1 / 3} \| + \mu_k\zeta_k \right).
\end{equation*}
Then, using the above bound in~\eqref{eq:boundtildeo} we have:
\begin{multline}\label{eq:prop1fromprop1}
   \| \vect y - \tilde{\mathcal O}_{f_1}( \mu_k, \vect x_{k + 1 / 3} ) \|^2 \leq \| \vect y - \vect x_{k + 1 / 3} \|^2 - 2\mu_k\bigl( f( \tilde{\mathcal O}_{f_1}( \mu_k, \vect x_{k + 1 / 3} ) ) - f( \vect y ) \bigr)\\
   {}+ \mu_k\bigl( \rho_1( \mu_k, k ) + \zeta_k( 2\| \vect y -\vect x_{k + 1 / 3} \| + \mu_k\zeta_k ) \bigr).
\end{multline}
That is, $\tilde{\mathcal O}_{f_1}$ satisfies Property~\ref{prop:propopt1} with $\rho_1$ replaced by
\begin{equation*}
   \tilde\rho_1( \mu_k ,k ) := \rho_1( \mu_k, k ) + \zeta_k( 2\| \vect y -\vect x_{k + 1 / 3} \| + \mu_k\zeta_k),
\end{equation*}
where we notice that the set of points $\vect y$ where Property~\ref{prop:propopt1} is applied in the convergence proofs is bounded if $\{ \vect x_k \}$ is bounded.

Given the above considerations, we are ready to provide the convergence results for Algorithm~\ref{algo:fiba}.
\begin{theorem}\label{theo:convFIBA}
   Assume $f_0$ is differentiable with Lipschitz-continuous gradient and has Lipschitz constant $L_0$. Suppose too that  $f_0$ has a bounded gradient and that $f_1$ has a bounded subgradient and $\{ f_1( \vect x_{k + 1 / 3} ) \}$ bounded from below. Assume $\{ \lambda_k \}$, $\{ \mu_k \}$ and $\{ \zeta_k \}$ are non-negative vanishing scalar sequences such that $\sum_{k = 0}^\infty \lambda_k = \infty$, $\lambda_k \leq 1/L_0$ (or each $\lambda_{i_k}$ satisfies~\eqref{eq:upperBoundLips}), $\sum_{k = 0}^\infty \mu_k = \infty$, and $\mu_k / \lambda_k \to 0$. Then, if $X_2 \neq \emptyset$, $\{\vect x_k\}$ is bounded, and $\mathcal O_{f_1}$ satisfies Properties~\ref{prop:propopt1}~and~\ref{prop:propoptbound} with $\rho_1( \mu_k, k ) \to 0$, we have
   \begin{equation*}
      \lim_{k \to \infty} d_{X_2}( \vect x_k ) = 0.
   \end{equation*}
\end{theorem}
\begin{proof}
   First let us notice that Algorithm~\ref{algo:fiba} can be written as
   \begin{equation*}
      \begin{split}
         \vect x_{k + 1 / 3} &{}:= \mathcal{G}_{f_0}( {\tilde\lambda}_k, \vect x_k )\text;\\
         \vect x_{k + 2 / 3} &{}:= \tilde{\mathcal{O}}_{f_1}( \mu_k, \vect x_{k + 1 / 3} )\text;\\
         \vect x_{k + 1} &{}:= \mathcal{P}_{X_0}( \vect x_{k + 2 / 3} )\text,
      \end{split}
   \end{equation*}
   where ${\tilde\lambda}_k := \lambda_{i_k}$. Since the construction of the algorithm guarantees that $i_k \leq k$, we have $\sum_{k = 0}^\infty {\tilde\lambda}_k = \infty$ and $\mu_k / {\tilde\lambda}_k \to 0$. Because ${\tilde\lambda_k} \leq 1/L_0$, \eqref{eq:fibaopt} holds and therefore, $\mathcal G_{f_0}$ satisfies Property~\ref{prop:propopt1} with $\rho_0( {\tilde\lambda}_k, k ) \equiv 0$. Furthermore, because~\eqref{eq:gradBoundDiff} and the algorithm definition, as already argued, we have $\| \vect x_{k} - \vect x_{k + 1 / 3} \| \to 0$. Also, if $\mathcal O_{f_1}$ satisfies Property~\ref{prop:propoptbound} so does ${\tilde{\mathcal O}}_{f_1}$. 
   Furthermore, as shown above, if $\mathcal O_{f_1}$ satisfies Property~\ref{prop:propopt1} so does ${\tilde{\mathcal O}}_{f_1}$, with the error term given by the factor multiplying $\mu_k$ in the second line of~\eqref{eq:prop1fromprop1}. Thus, the assumed boundedness of $\{ \vect x_k \}$ and of $f_1$ ensure that Proposition~\ref{prop:CX} can be applied so that
   \begin{equation*}
      \lim_{k \to \infty} d_{X_1}( \vect x_k ) = 0.
   \end{equation*}

   Now, because $\vect x_{k + 1 / 3} \in X_0$, we then have $f_0( \vect x_{k + 1 / 3} ) \geq f_0^*$. Furthermore, because of the boundedness assumptions and of $\mu_k \to 0$, it is possible to see that $\tilde\rho_1( \mu_k, k ) \to 0$. Therefore, Proposition~\ref{prop:CX2} can be applied, which leads to the desired conclusion.\qed
\end{proof}

\subsection{Incremental Algorithms for Non-Differentiable Problems}\label{subsec:algoinc}

Here we specialize Algorithm~\eqref{eq:algo} to the case where $f_0$ is the sum of several non-differentiable convex functions:
\begin{equation*}
   f_0 := \sum_{i = 1}^m f_0^i.
\end{equation*}
In this situation we propose the use, for the primary optimization problem, of the incremental subgradient operator, denoted as $\mathcal I_f : \mathbb R^n \times \mathbb R \to \mathbb R^n$, given by:
\begin{equation*}
   \begin{split}
      \vect x^{(1)} & {}= \vect x\\
      \vect x^{(i + 1)} & {}= \vect x^{( i )} - \lambda\tilde\nabla f^i( \vect x^{(i)} )\quad i = 1, 2, \dots, m\\
      \mathcal I_f( \lambda, \vect x ) &{}= \vect x^{( m + 1 )}.
   \end{split}
\end{equation*}

Incremental operators are well known for its fast initial convergence rate and, accordingly, several variations of it have been thoroughly analyzed in the literature~\cite{soz98,sol98,neb01,bhg07,dey01,bet00,ber97}. We will use here the result~\cite[Lemma~2.1]{neb01}:
\begin{lemma}\label{lemm:incsub}
   Assume the subgradients of the convex functions $f_0^i$ are bounded in the following sense:
   \begin{equation}\label{eq:subgradBoundInc}
      \forall \vect x \in \mathbb R^n\quad\text{and}\quad \forall\vect v \in \partial f_0^i( \vect x ), \quad \| \vect v \| \leq C_i, \quad i\in \{1, 2, \dots, m\}.
   \end{equation}
   Then, the incremental subgradient operator satisfies, for every $\lambda \in \mathbb R_+$, and $\vect y, \vect x \in \mathbb R^n$:
   \begin{equation}\label{eq:incsub}
      \| \mathcal I_{f_0}( \lambda, \vect x ) - \vect y \|^2 \leq \| \vect x - \vect y \|^2 - 2\lambda \bigl( f_0( \vect x ) - f_0( \vect y ) \bigr) + \lambda^2\bigl( \sum_{i = 1}^mC_i \bigr)^2\text,
   \end{equation}
   where $f_0 := \sum_{i = 1}^m f_0^i$.
\end{lemma}

Now, notice that the boundedness condition on the subdifferentials leads, for every $\vect x \in \mathbb R^n$, to
\begin{equation}\label{eq:incsubbound}
   \| \mathcal I_{f_0}( \lambda, \vect x ) - \vect x \| \leq \lambda\sum_{i = 1}^m C_i,\quad\text{and}\quad \tilde\nabla f_0( \vect x ) \in \partial f_0( \vect x ) \Rightarrow \| \nabla f_0( \vect x ) \| \leq \sum_{i = 1}^m C_i.
\end{equation}
Thus, applying Proposition~\ref{prop:oneIsOther} we are lead to the following result:
\begin{corollary}
   Assume the subgradients of the convex functions $f_0^i$ satisfy~\eqref{eq:subgradBoundInc}. Then, the incremental subgradient operator satisfies, for every $\lambda \in \mathbb R_+$, and $\vect y, \vect x \in \mathbb R^n$:
   \begin{equation*}
      \| \mathcal I_{f_0}( \lambda, \vect x ) - \vect y \|^2 \leq \| \vect x - \vect y \|^2 - 2\lambda \bigl( f_0\bigl( \mathcal I_{f_0}( \lambda, \vect x ) \bigr) - f_0( \vect y ) \bigr) + 3\lambda^2\bigl( \sum_{i = 1}^mC_i \bigr)^2\text,
   \end{equation*}
   where, again, $f_0 := \sum_{i = 1}^m f_0^i$.
\end{corollary}

The second algorithm we propose in this work will be called \textsc{iiba}, from Incremental Iterative Bilevel Algorithm, and is described in Algorithm~\ref{algo:iiba} below. For this algorithm we have the following convergence result.

\begin{algorithm}
   \begin{algorithmic}[1]
      \Require{$\vect x_{0}$, $\{\lambda_k\}$, $\{\mu_k\}$}%
      \Statex%
      \State{Initialization: $k \leftarrow 0$}%
      \Statex%
      \Repeat
         \Statex%
         \State{$\vect x_{k + 1 / 3} = \mathcal I_{f_0}( \vect x_k )$}
         \State{$\vect x_{k + 2 / 3} = \mathcal O_{f_1}( \vect x_{k + 1 / 3}, \mu_k )$}%
         \State{$\vect x_{k + 1} = \proj_{X_0}( \vect x_{k + 2/ 3} )$}%
         \State{$k \leftarrow k + 1$}%
         \Statex%
      \Until{convergence is reached}
   \end{algorithmic}
   \caption{Incremental Iterative Bilevel Algorithm}\label{algo:iiba}
\end{algorithm}

\begin{theorem}\label{theo:convIIBA}
   Assume that $f_0$ is of the form $f_0 := \sum_{i = 1}^m f_0^i$ and satisfies~\eqref{eq:subgradBoundInc}, that $f_1$ has a bounded subgradient, and that $\{ f_1( \vect x_{k + 1 / 3} ) \}$ is bounded from below. Assume $\{ \lambda_k \}$ and $\{ \mu_k \}$ are non-negative vanishing scalar sequences such that $\sum_{k = 0}^\infty \lambda_k = \infty$, $\sum_{k = 0}^\infty \mu_k = \infty$, $\mu_k / \lambda_k \to 0$ and $\lambda_k^2 / \mu_k \to 0$. Then, suppose $X_2 \neq \emptyset$ and $X_2$ is bounded (or $\{\vect x_k\}$ is bounded), and $\mathcal O_{f_1}$ satisfies Properties~\ref{prop:propopt1}~and~\ref{prop:propoptbound} with $\rho_1( \mu_k ,k ) \to 0$. Then, the sequence $\{ \vect x_k \}$ generated by Algorithm~\ref{algo:iiba} satisfies
   \begin{equation*}
      \lim_{k \to \infty} d_{X_2}( \vect x_k ) = 0.
   \end{equation*}
\end{theorem}
\begin{proof}
   Notice that Lemma~\ref{lemm:incsub} together with Proposition~\ref{prop:oneIsOther} and the subgradient boundedness assumption imply that $\mathcal I_{f_0}$ satisfies the desired Property~\ref{prop:propopt1}. Also, $\lambda_k \to 0$ implies, together with the subgradient boundedness assumption, that $\| \vect x_k - \vect x_{k + 1/3} \| \to 0$ and that $\rho_0( \lambda_k, k ) \to 0$, the latter because of Lemma~\ref{lemm:incsub}. 
   Therefore, Proposition~\ref{prop:CX} implies that
   \begin{equation*}
      d_{X_1}( \vect x_k ) \to 0.
   \end{equation*}

   Now, notice the subgradient boundedness assumption and~\eqref{eq:incsubbound} imply that $[ f_0( \vect x_k ) - f_0( \vect x_{k + 1 / 3} ) ]_+  = O( \lambda_k )$, and, therefore, since $\vect x_k \in X_0$ for $k > 0$, we have $[ f_0^* - f( \vect x_{k + 1 / 3} ) ]_+  = O( \lambda_k )$. Thus, $\lambda_k^2 / \mu_k \to 0$ implies $\lambda_k[ f_0^* - f( \vect x_{k + 1 / 3} ) ]_+ / \mu_k \to 0$. Furthermore, since, by~\eqref{eq:incsub}, $\rho_0( \lambda_k, k ) = O( \lambda_k )$, $\lambda_k^2 / \mu_k$ also implies $\lambda_k\rho_0( \lambda_k, k ) / \mu_k \to 0$. So, finally, Proposition~\ref{prop:CX2} can be applied, which proves the result.\qed

\end{proof}

\subsection{Stopping Criterion}

Here we devise a stopping criterion for the proposed bilevel methods, based on inequalities~\eqref{eq:approx_f0}~and~\eqref{eq:diffestf1Final} coupled to the following Lemma:
\begin{lemma}\label{lemm:ratioOfSums}
   Let $\{ a_k \}$ and $\{ b_k \}$ be non-negative sequences such that $a_k / b_k \to 0$ and $\sum_{k = 0}^\infty b_k = \infty$. Then
   \begin{equation*}
      \lim_{n \to \infty} \frac{\sum_{k = 0}^n a_k}{\sum_{k = 0}^n b_k} = 0.
   \end{equation*}
\end{lemma}
\begin{proof}
   Choose any $\alpha > 0$ and let $k_0$ be such that $a_k / b_k \leq \alpha$ for every $k \geq k_0$. Then
   \begin{equation*}
      \begin{split}
         \lim_{n \to \infty} \frac{\sum_{k = 0}^n a_k}{\sum_{k = 0}^n b_k} &{}= \lim_{n \to \infty} \frac{\sum_{k = 0}^{k_0 - 1} a_k + \sum_{k = k_0}^n a_k}{\sum_{k = 0}^n b_k}\\
            &{}= \lim_{n \to \infty} \frac{\sum_{k = k_0}^n a_k}{\sum_{k = 0}^n b_k}.
      \end{split}
   \end{equation*}
   But
   \begin{equation*}
      \begin{split}
         \frac{\sum_{k = k_0}^n a_k}{\sum_{k = 0}^n b_k} \leq \frac{\sum_{k = k_0}^n a_k}{\sum_{k = k_0}^n b_k} = \frac{\sum_{k = k_0}^n ( a_k / b_k ) b_k}{\sum_{k = k_0}^n b_k} \leq \alpha.
      \end{split}
   \end{equation*}
   Therefore, $\lim_{\to\infty} \left| \sum_{k = 0}^n a_k \middle/ \sum_{k = 0}^n b_k \right| \leq \alpha$ for any $\alpha > 0$.\qed
\end{proof}

In order to explain the stopping criterion, we resort to the concept of best-so-far iteration. Let us denote as $\phi_i^{k_0, k}$, for $i \in \{ 0, 1 \}$ and $k_0 \leq k$ integers in $\{ 0, 1, \dots \}$, the smallest value in the set $\{ f_i( \vect x_{k_0} ), f_i( \vect x_1 ), \dots, f_i( \vect x_k ) \}$. For the special case $k_0 = 0$, we simplify the notation by $\phi_i^k := \phi_i^{0, k}$. Then, successive application of~\eqref{eq:approx_f0} together with $\phi_i^k \leq f_i( \vect x_j )$ for all $j \leq k$ leads to
\begin{multline*}
   \phi_0^k - f_0^* \leq \frac{\| \vect x_0 - \vect x^* \|^2}{\beta\sum_{i = 0}^k\lambda_i} + \frac{\sum_{i = 0}^k\lambda_i\bigl( \rho_0( \lambda_i, i ) + \beta M\| \vect x_{i + 1 / 3} - \vect x_{i} \| \bigr)}{\beta\sum_{i = 0}^k\lambda_i}\\
   {}+ \frac{\sum_{i = 0}^k\mu_i\bigl( \beta N + \rho_1( \mu_i, i ) \bigr)}{\beta\sum_{i = 0}^k\lambda_i} =: \sigma_0^k.
\end{multline*}
Therefore, under the hypothesis of Proposition~\ref{prop:CX}, Lemma~\ref{lemm:ratioOfSums} ensures that the right-hand side of the above inequality vanishes as the iterations proceed. Thus, the quantity $\sigma_0^k$
can be used as a measure of convergence to $X_1$ bounding the difference between the best $f_0$ function value to the optimal $f_0^*$, as long as it is possible to estimate the distance $\| \vect x_0 - \vect x^* \|$, the behavior of the error terms and the subgradient bounding constants. Notice that, in principle, the constant $N$ require knowledge of the optimal value in this case, but it can be replaced by an upper bound for it, which should not be difficult to obtain in many cases.

We can use a very similar reasoning in order to estimate optimality of the secondary objective function too. Applying~\eqref{eq:diffestf1Final} repeatedly and recalling that $\phi_1^{k_0, k} \leq f_1( \vect x_i )$ for every $i \in \{ k_0, k_0 + 1, \dots, k \}$ we have
\begin{multline*}
   \phi_1^{k_0, k} - f_1^* \leq \frac{\| \vect x_{k_0} - \vect x^* \|^2}{\beta\sum_{i = k_0}^k\mu_i} + \frac{\sum_{i = k_0}^k\lambda_i\bigl( \rho_0( \lambda_i, i ) + \beta [f_0^* - f_0( \vect x_{k + 1/3} )]_+\bigr)}{\beta\sum_{i = k_0}^k\mu_i}\\
   {}+ \frac{\sum_{i = k_0}^k\mu_i\bigl( \rho_1( \mu_i, i ) + \beta M\| \vect x_{i + 2 / 3} - \vect x_{i} \|\bigr)}{\beta\sum_{i = k_0}^k\mu_i} =: \sigma_1^{k_0, k}.
\end{multline*}
Now, under the hypothesis of Proposition~\ref{prop:CX2}, Lemma~\ref{lemm:ratioOfSums} ensures that for any $k_0$, we have $\lim_{k \to \infty}\sigma_0^{k_0, k} = 0$.

We can now describe how to stop the algorithm at a non-negative integer $\kappa_{\epsilon_0, \epsilon_1}$ such that we have
\begin{equation*}
   f_0( \vect x_{\kappa_{\epsilon_0, \epsilon_1}} ) - f_0^* \leq \epsilon_0\quad\text{and}\quad f_1( \vect x_{\kappa_{\epsilon_0, \epsilon_1}} ) - f_1^* \leq \epsilon_1\text,
\end{equation*}
for any pair of positive numbers $\epsilon_0$ and $\epsilon_1$. Let us consider the following procedure:
\begin{enumerate}
   \item Iterate the algorithm until $\sigma_0^k \leq \epsilon_0$;
   \item $k_0 \leftarrow k$, $\kappa \leftarrow k$;
   \item Iterate the algorithm until $\sigma_1^{\kappa, k} \leq \epsilon_1$;\label{item:stopf1}
   \item Let $k_1 \geq \kappa$ be such that $f_1( \vect x_{k_1} ) = \phi_1^{\kappa, k}$;
   \item If $f_0( \vect x_{k_1} ) \leq \phi_0^{k_0}$: STOP;\label{item:stopf0}
   \item $\kappa \leftarrow k$; Go to step~\ref{item:stopf1}.
\end{enumerate}
Notice that once the procedure has stopped, then
\begin{equation*}
   f_0( \vect x_{k_1} ) - f_0^* \leq \phi_0^{k_0} - f_0^* \leq \sigma_0^k \leq \epsilon_0\text,
\end{equation*}
and
\begin{equation*}
   f_1( \vect x_{k_1} ) - f_0^* = \phi_1^{\kappa, k} - f_0^* \leq \sigma_1^{\kappa, k} \leq \epsilon_1\text.
\end{equation*}
The procedure indeed stops because since Proposition~\ref{prop:CX} ensures convergence in norm to $X_1$, mild subgradient boundedness assumptions will guarantee also that $f_0( \vect x_k ) \to f_0^*$. Therefore we know that for large enough $k$, there holds $f_0( \vect x_{k} ) \leq \phi_0^{k_0}$. That is, the condition in step~\ref{item:stopf0} will eventually be satisfied as the iterations proceed.

\section{Application Problem Presentation}

We have performed experiments involving tomographic image reconstruction. In tomography, the idealized problem is to reconstruct a function $\mu : \mathbb R^2 \to \mathbb R$ given the values of its integrals along straight lines, that is, given its \emph{Radon transform} denoted as $\mathcal R[ \mu ]$ and defined by the following equality:
\begin{equation*}
   \mathcal R[ \mu ]( \theta, t ) := \int_{\mathbb R} \mu\left( t\left(\begin{smallmatrix}\cos\theta\\ \sin\theta\end{smallmatrix}\right) + s \left(\begin{smallmatrix}-\sin\theta\\ \cos\theta\end{smallmatrix}\right)\right) \mathrm ds.
\end{equation*}
Figure~\ref{fig:rad} (adapted from~\cite{hcc14}) brings a graphical representation of this definition.

\begin{figure}
   \centering%
   \newcommand{\shepplogan}
{%
   \begin{scope}[line width=0pt]
      \path[color=black!100,draw,fill] (0,0)       ellipse (0.69 and 0.92);
      \path[color=black!20,draw,fill]  (0,-0.0184) ellipse (0.6624 and 0.874);

      \path[color=black!30,draw,fill] (0,0.35)       ellipse (0.21 and 0.25);
      \path[color=black!30,draw,fill] (0,-0.1)       ellipse (0.046 and 0.046);
      \path[color=black!30,draw,fill] (-0.08,-0.605) ellipse (0.046 and 0.023);
      \path[color=black!30,draw,fill] (0,-0.606)     ellipse (0.023 and 0.023);
      \path[color=black!30,draw,fill] (0.06,-0.605)  ellipse (0.023 and 0.046);
      \path[color=black!30,draw,fill] (0.06,-0.605)  ellipse (0.023 and 0.046);

      \path[color=black!0,draw,fill,xshift=0.22\grfxunit,rotate=-18] (0,0) ellipse (0.11 and 0.31);
      \path[color=black!0,draw,fill,xshift=-0.22\grfxunit,rotate=18] (0,0) ellipse (0.16 and 0.41);
      \begin{scope}
         \path[clip,xshift=-0.22\grfxunit,rotate=18] (0,0)    ellipse (0.16 and 0.41);
         \path[color=black!10,draw,fill]             (0,0.35) ellipse (0.21 and 0.25);
         \path[color=black!10,draw,fill]             (0,-0.1) ellipse (0.046 and 0.046);
      \end{scope}
   \end{scope}

   \path[color=black!30,draw,fill]    (0,0.1)  ellipse (0.046 and 0.046);
   \begin{scope}
      \path[clip]                     (0,0.1)  ellipse (0.046 and 0.046);
      \path[color=black!40,draw,fill] (0,0.35) ellipse (0.21 and 0.25);
   \end{scope}
}%
   \input{figradonteta.tex}%
   \def\tlinha{-0.6}%
   \setlength{\grftotalwidth}{0.45\columnwidth}%
   \setlength{\grfticksize}{0.5\grfticksize}%
   \small{\ }\hfill%
   \begin{grfgraphic}{%
      \def\grfxmin{-2.05}\def\grfxmax{1.5}%
      \def\grfymin{-1.5}\def\grfymax{2.05}%
   }%
      \begin{scope}[>=stealth,style=grfaxisstyle,<->]%
         \shepplogan%
         \draw (-1.5,0) -- (1.5,0);%
         \draw (0,-1.5) -- (0,1.5);%
         \foreach \i in {-1,1}%
         {
            \draw[style=grftickstyle,-] (\i\grfxunit,-\grfticksize) -- (\i\grfxunit,\grfticksize);%
            \draw[style=grftickstyle,-] (-\grfticksize,\i\grfyunit) -- (\grfticksize,\i\grfyunit);%
         }
         \begin{scope}[rotate=\teta]
            \draw (-1.5,0) -- (1.5,0);
            \foreach \i in {-1,1}
               \draw[style=grftickstyle,-] (\i\grfxunit,-\grfticksize) -- (\i\grfxunit,\grfticksize);
            \draw[dashed,dash phase=-0.005\grfyunit] (\tlinha,-1.5) -- (\tlinha,1.5);
            \fill (\tlinha,0) node[anchor=north,inner sep=\grflabelsep] {\scriptsize$t$} circle (0.025cm);
            \draw[-] (\tlinha\grfxunit,0.3em) -| (\tlinha\grfxunit - 0.3em,0pt);
            \fill (\tlinha\grfxunit - 0.15em,0.15em) circle (0.025cm);
         \end{scope}
         \begin{scope}[rotate=\teta,yshift=1.5\grfyunit]
            \draw[line width=0.025cm] plot file {figradon.data};
            \draw[-]  (1.5,0) -- (-1.5,0) node[anchor=north,rotate=\teta,inner sep=\grflabelsep] {\scriptsize$t$};
            \draw[->] (0,0)    -- (0,0.7) node[anchor=west,rotate=\teta,inner sep=\grflabelsep]  {\scriptsize$\mathcal R[ \mu ](\theta,t)$};
         \end{scope}
         \def\rad{0.075}
         \FPupn{\cpt}{0.552285 \rad{} * 90 \teta{} / *}
         \path (\teta:\rad) ++(\teta - 90:\cpt) node (a) {};
         \draw[-] (0,0) -- (\rad,0) .. controls +(0,\cpt) and (a) .. (\teta:\rad) -- cycle;
         \draw[style=grftickstyle,-,rotate=\tetameio,xshift=\rad\grfxunit] (-\grfticksize,0pt) -- (\grfticksize,0pt);
         \path[rotate=\tetameio,xshift=0.3em] (\rad,0) node[anchor=west,inner sep=0pt] {\scriptsize$\theta$};
      \end{scope}
   \end{grfgraphic}\hfill%
   \begin{grfgraphic}[1.125]{%
      \grfyaxis[R]{[]-1;0;1[]}{[]\tiny$-1$;\tiny$0$;\tiny$1$[]}%
      \grfylabel{\footnotesize$t$}%
      \grfxaxis[R]{[]0;1.57;3.14[]}{[]\tiny$0$;\tiny$\frac\pi2$;\tiny$\pi$[]}%
      \grfxlabel{\footnotesize$\theta$}%
      \def\grfxmin{0}%
      \def\grfxmax{3.14}%
      \grfwindow%
   }%
      \node[anchor=north west,inner sep=0pt] at (0,1)
      {\includegraphics[width=3.14\grfxunit,height=2\grfyunit]
      {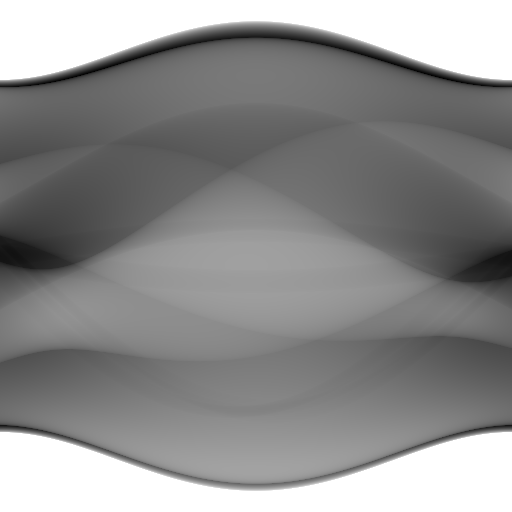}};%
   \end{grfgraphic}\hfill{\ }%
   \caption{Left: schematic representation of the Radon transform. In
   the definition, $\theta$ is the angle between the normal to
   the integration path and the horizontal axis, while $t$ is the
   line of integration's displacement from the origin.
   Right: Image of the Radon transform of the image shown on the
   left in the $\theta \times t$ coordinate system.}\label{fig:rad}%
\end{figure}

Data was obtained by transmitting synchrotron radiation through samples of eggs taken from a fish of the species \emph{Prochilodus lineatus} collected at the Madeira River`s bed, immersed in distilled water inside a capillary test tube. Data acquisition was performed at the Brazilian National Synchrotron Light Source (\textsc{lnls})\footnote{\url{http://lnls.cnpem.br/}}.

In a transmission tomography setup~\cite{kas88,nat86,her80} like the one we have used, the value of the line integral is estimated through the emitted to detected intensity ratio according to Beer--Lambert law:
\begin{equation*}
   \frac{I_e}{I_d} = e^{\int_{L}\mu( \vect x )\mathrm ds},
\end{equation*}
where $I_e$ is the emitted intensity, $I_d$ is the detected intensity, $\mu$ gives the linear attenuation factor of the imaged object at each point in space, and $L$ is the straight line connecting detector to emitter. While in this case the reconstruction problem is essentially bi-dimensional, a simultaneous acquisition of a radiography of $2048$ parallel slices of the object to be imaged is made at each angle, which enables volumetric reconstruction, if desired, by the stacking of several bi-dimensional reconstructions. After each plain \textsc{x}-ray imaging, the sample is rotated and new samples of the Radon transform are estimated in the same manner at a new angle. Figure~\ref{fig:experimental_Radon} depicts the process of assembling the $2048 \times 200$ data array, which will be used for a slice reconstruction, from the $200$ images of size $2048 \times 2048$.

In our application the imaged subject is sensitive in a way such that a overly long exposure time under a low energy \textsc{x}-ray beam may overheat or otherwise physically damage the sample. Therefore, because the exposure time for good radiographies under a monochromatic beam at LNLS' facilities was experimentally found to be at least $20$ seconds, the Radon Transform was sampled at only $200$ evenly spaced angles covering the interval $[-\pi, 0]$, a relatively small number if we are willing to reconstruct full resolution $2048 \times 2048$ images from this data. In this case, it is likely that problem~\eqref{eq:leastsq} will have many solutions and we need to select one of these, therefore the need of a bilevel model arises.

\begin{figure}
   \input{egg_illustration_data.tex}%
   \centering%
   \setlength{\grftotalwidth}{\textwidth}%
   \begin{grfgraphic}{%
      \def\grfxmin{0.0}\def\grfxmax{3.0}%
      \def\grfymin{-0.5}\def\grfymax{0.5}%
   }%
      \node[inner sep=0pt] at (0.5,0.0) {\includegraphics[width=\grfxunit]{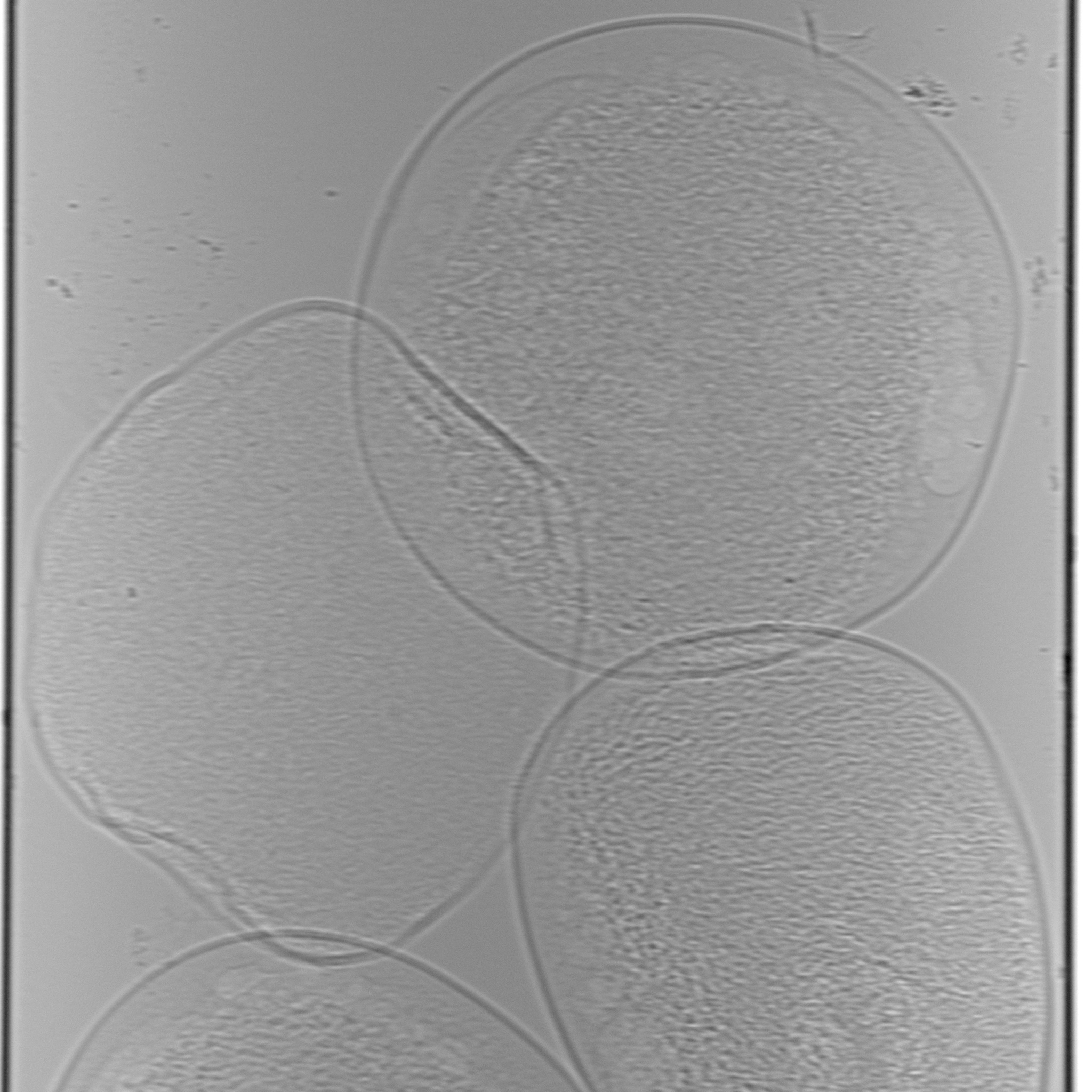}};%
      \node[inner sep=0pt] at (1.5,0.0) {\includegraphics[width=\grfxunit]{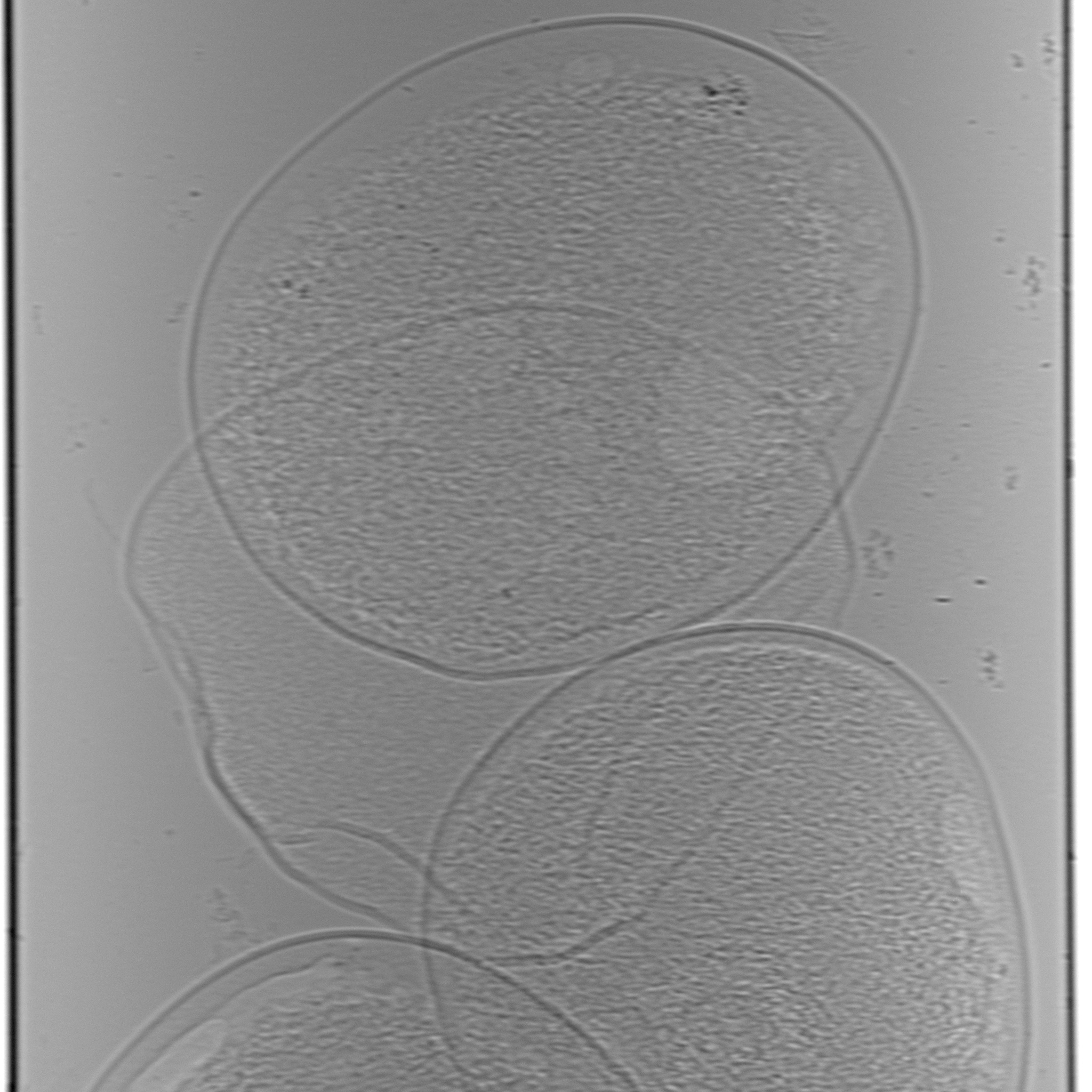}};%
      \node[inner sep=0pt] at (2.5,0.0) {\includegraphics[width=\grfxunit]{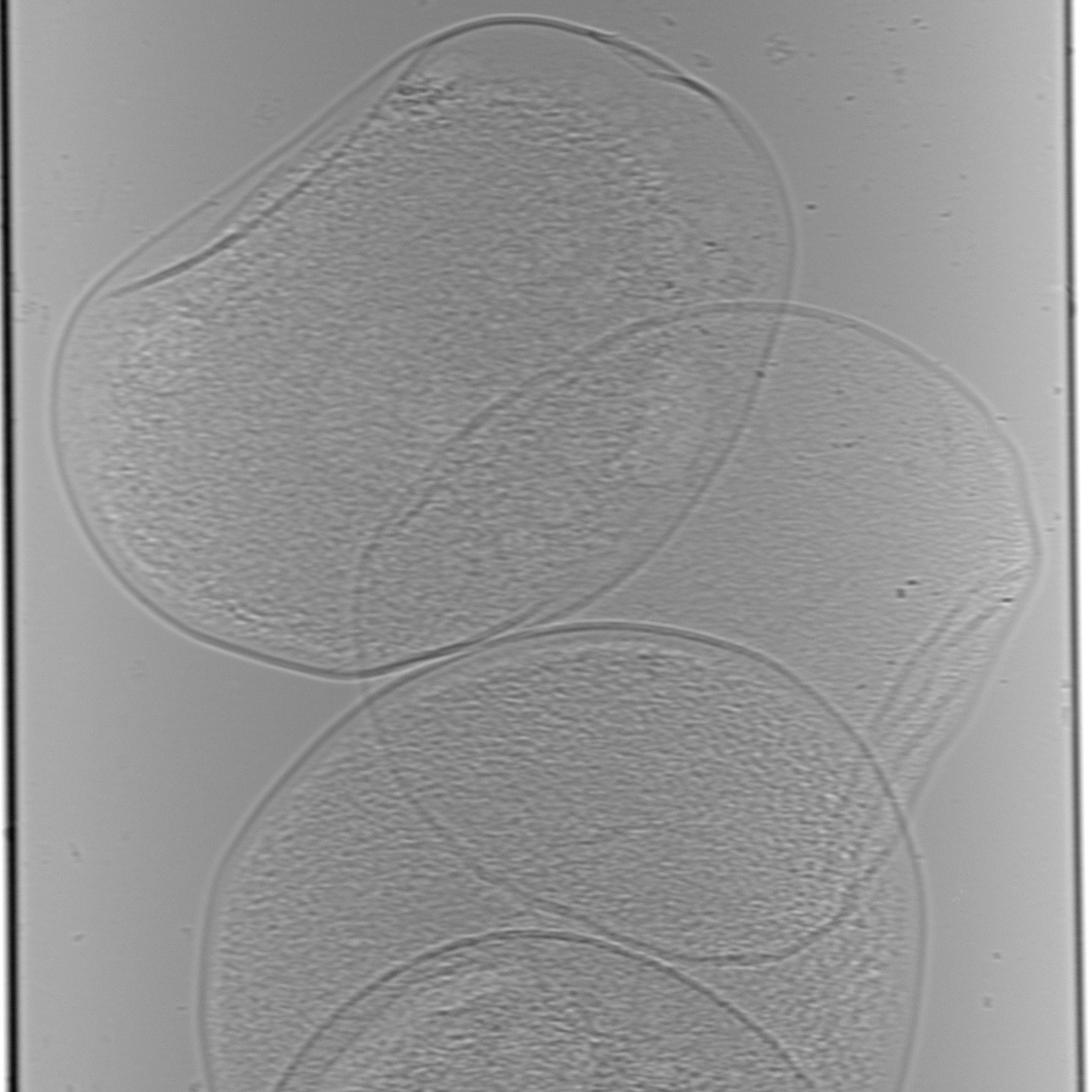}};%
      \draw[grfaxisstyle] (0.0,-0.5) rectangle (1.0,0.5);%
      \draw[grfaxisstyle] (1.0,-0.5) rectangle (2.0,0.5);%
      \draw[grfaxisstyle] (2.0,-0.5) rectangle (3.0,0.5);%
      \draw[grfaxisstyle,red,opacity=0.5] (0.0,\sliceHeight) -- (1.0,\sliceHeight);%
      \draw[grfaxisstyle,green,opacity=0.5] (1.0,\sliceHeight) -- (2.0,\sliceHeight);%
      \draw[grfaxisstyle,blue,opacity=0.5] (2.0,\sliceHeight) -- (3.0,\sliceHeight);%
   \end{grfgraphic}\\[5pt]%
   \begin{grfgraphic}[0.6]{%
      \grfxaxis{[]0.0;1.0;2.0;3.0[]}{[]\small$0.0$;\small$-\frac\pi3$;\small$-\frac{2\pi}3$;\small$-\pi$[]}%
      \grfyaxis[r]{[]-1.0;0.0;1.0[]}{[]\small$-0.38$;\small$0.0$;\small$0.38$[]}%
   }%
      \node[inner sep=0pt] at (1.5,0.0) {\includegraphics[width=3\grfxunit,height=2\grfyunit]{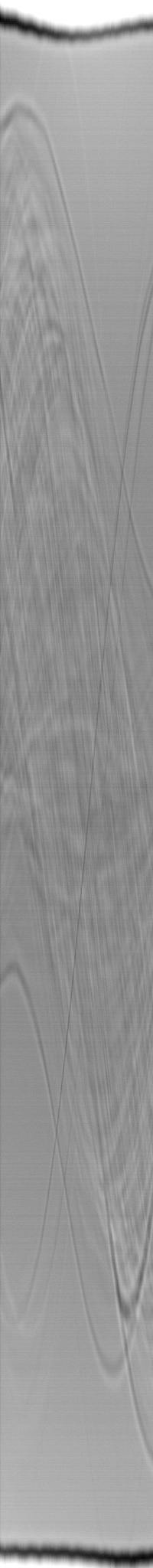}};%
      \draw[grfaxisstyle] (0.0,-1.0) rectangle (3.0,1.0);%
      \draw[grfaxisstyle,red,opacity=0.5] (\projectionAPos,-1.0) -- (\projectionAPos,1.0);%
      \draw[grfaxisstyle,green,opacity=0.5] (\projectionBPos,-1.0) -- (\projectionBPos,1.0);%
      \draw[grfaxisstyle,blue,opacity=0.5] (\projectionCPos,-1.0) -- (\projectionCPos,1.0);%
   \end{grfgraphic}%
   \caption{Assembly of the Radon Transform. Top: three of the $200$ radiographic images used, each of which has $2048 \times 2048$ pixels and depicts a square region of area $0.76 \times 0.76$mm${^2}$. Bottom: the $n^{\text{\scriptsize th}}$ row of the $i^{\text{\scriptsize th}}$ image has samples of $\mathcal R[ g_n ]( \theta_i, \cdot )$, that is, a column in the representation of $\mathcal R[ g_n ]$ in the $\theta \times t$ plane, where $g_n : \mathbb R^2 \to \mathbb R$ gives the linear attenuation factor at each point in the $n^{\text{\scriptsize th}}$ slice to be reconstructed. The colored solid lines depict the position of the radiographies' rows in the resulting sinogram.}\label{fig:experimental_Radon}%
\end{figure}

\subsection{Primary Objective Functions}

Assuming the original image $g : \mathbb R^2 \to \mathbb R_+$ lies in a finite dimensional vector space generated by some basis $\{ g^1, g^2, \dots, g^n \}$ and recognizing that the number of measurements is always finite in practice, one can reduce the problem of tomographic reconstruction to a linear system of equations:
\begin{equation}\label{eq:linsys}
   R\vect x = \vect b,
\end{equation}
where the elements $r_{ij}$ of the matrix $R$ are given by
\begin{equation}\label{eq:linsys}
   \mathcal R[g^j]( \theta_i, t_i )
\end{equation}
and the elements $b_i$ of the vector $\vect b$ are the corresponding experimental data, that is, $b_i$ is an approximate sample of $\mathcal R[g]( \theta_i, t_i )$, where $g = \sum_{i = 1}^nx_ig^i$ is the desired image. Because actual microtomographic data from synchrotron illumination will always contain errors (either from micrometric misalignment of the experimental setup, small dust particles in the emitter-detector path, or from statistical fluctuation of photon emission and attenuation), the above linear system of equations may not have a solution.

\subsection{Differentiable Primary Objective Function}

System~\eqref{eq:linsys} can be replaced by a constrained least squares problem, in order to alleviate the likely lack of consistency:
\begin{equation}\label{eq:leastsq}
   X_1 = \argmin_{\vect x \in \mathbb R_+^n} q( \vect x ) := \frac12\| R\vect x - \vect b \|^2.
\end{equation}
Therefore, in this case we will be lead to a bilevel problem of the form~\eqref{eq:bilevel} with $X_0 = \mathbb R_+^n$.


Another option for a feasibility set would be the bounded box $\{ \vect x \in \mathbb R^n : 0 \leq x_i \leq u \}$ for some $u > 0$. This may be a sensible idea if the maximum attenuation factor is known beforehand and we could thus include this information into the model. Because we do not make this kind of boundedness imposition, by using the least squares function we are at the risk of violating certain (sub)gradient boundedness assumption. We therefore use another continuously differentiable convex function $f_0$, which has an uniformly bounded gradient and is such that
\begin{equation}\label{eq:probEquiv}
   \argmin_{\vect x \in X_0} f_0( \vect x ) = \argmin_{\vect x \in X_0} q( \vect x ).
\end{equation}
In order to build such function, take any $\tilde{\vect x} \in X_0$ (such as, in our case, $\tilde{\vect x} = \vect 0$) and set
\begin{equation*}
   \Delta := \| R\tilde{\vect x} - \vect b \| > 0,
\end{equation*}
then define
\begin{equation*}
   f_0( \vect x ) := \frac12\sum_{i = 1}^m h( \langle r_i, \vect x \rangle - b_i ),
\end{equation*}
where $r_i^T$ is the $i$-th row of $R$ and the function $h : \mathbb R \to \mathbb R_+$ is given as
\begin{equation*}
   h( x ) := \begin{cases}
                      x^2 & \text{if}\quad |x| < \Delta\\
                      2\Delta|x| - \Delta^2 & \text{otherwise.}
                   \end{cases}
\end{equation*}
Notice that if $q(\vect x) \leq q( \tilde{\vect x} ) = 1/2\Delta^2$, then $f_0( \vect x ) = q( \vect x )$ and, furthermore, if $f_0( \vect x ) \neq q( \vect x )$, then $f_0( \vect x ) \geq 1/2\Delta^2 = q( \tilde{\vect x} )$. Therefore, \eqref{eq:probEquiv} holds.

We need to show that the uniformly bounded continuous differentiability and convexity claims hold, but these facts follow from convexity and continuous differentiability of $h$. Its derivative can be computed to be
\begin{equation*}
   h'( x ) := \begin{cases}
                 2x & \text{if}\quad |x| < \Delta\\
                 2\Delta\sign(x) & \text{otherwise,}
              \end{cases}
\end{equation*}
with $\sign : \mathbb R \to 2^{\mathbb R}$ defined as
\begin{equation*}
   \sign( x ) := \begin{cases}
                     -1      & \text{if}\quad x < 0\\
                     [-1, 1] & \text{if}\quad x = 0\\
                      1      & \text{if}\quad x > 0.
                 \end{cases}
\end{equation*}
Thus, the derivative of $h$ is uniformly bounded by $| h'( x ) | \leq 2\Delta$, is continuous as long as we use $\Delta > 0$, and, since $h'( x )$ is nondecreasing, $h$ is convex. Notice that $h$ is the well known Huber function~\cite{hub64}.

%

In order to simplify the notation, we introduce the function $\vect h' : \mathbb R^m \to \mathbb R^m$ given as:
\begin{equation*}
   \vect h'( \vect x ) := \left(\begin{matrix}
                                   h'( x_1 )\\
                                   h'( x_2 )\\
                                   \vdots\\
                                   h'( x_m )
                               \end{matrix}\right).
\end{equation*}
Under this notation, a straightforward calculation leads to
\begin{equation*}
   \nabla f_0( \vect x ) =\frac12 R^T\vect h'( R\vect x - \vect b ).
\end{equation*}
The computationally expensive parts of this expression are products of the form $R\vect x$ and $R^T\vect b$. In our case, because of image resolution and dataset size, as explained in the previous Subsection, matrix $R$ has dimensions of over $4\cdot 10^5$ lines by $4\cdot10^6$ columns.

In order to obtain a computationally viable algorithm, these matrix-vector products must be amenable to fast computation, and, in fact, there is a very effective approach to it based in a relation between the one-dimensional Fourier Transform of the data and one ``slice'' of the two-dimensional Fourier Transform of the original image. Recall that we denote our desired, unknown, image as $g : \mathbb R^2 \to \mathbb R_+$ and define the projection of $f$ at an angle $t$ as
\begin{equation*}
   p_\theta( t ) := \mathcal R[ g ]( t ).
\end{equation*}
Now, using the hat notation for the Fourier transform of an integrable function $f : \mathbb R^n \to \mathbb C$, as defined as follows, where $\jmath := \sqrt{-1}$:
\begin{equation*}
   \hat f( \vect \omega ) := \mathcal F[ f ]( \vect \omega ) := \int_{\mathbb R^n} f( \vect x ) e^{-\jmath 2\pi\langle \vect\omega, \vect x \rangle} \mathrm d\vect x\text,
\end{equation*}
there holds the so called Fourier-slice theorem~\cite{kas88,nat86}:
\begin{equation*}\label{eq:fourierSlice}
   \hat p_\theta( \omega ) = \hat g\left( \omega\left( \begin{smallmatrix}
                                                            \cos\theta\\
                                                            \sin\theta
                                                         \end{smallmatrix}
                                            \right)\right).
\end{equation*}

This result implies that the action of $R$ can be evaluated first in the Fourier space, an operation which can be computed efficiently by Non-uniform Fast Fourier Transforms (\textsc{nfft})~\cite{fou03} and later translated back to the original feature space by regular Fast Fourier Transforms (\textsc{fft}). We have used the \texttt{pynfft} binding for the \texttt{nfft3} library~\cite{kkp09} in order to compute samples of $\hat p_\theta$ following~\eqref{eq:fourierSlice}, and then used regular inverse \textsc{fft} routines from \texttt{numpy} for the final computation from the Fourier back to the feature space. Fast evaluation of the transpose operation is immediately available from the same set of libraries using the \textsc{fft} from \texttt{numpy} and the transpose of the \textsc{nfft} available from \texttt{nfft3}.

\subsection{Non-Differentiable Primary Objective Function}

Another option to circumvent the non-consistency of~\eqref{eq:linsys} under non-negativity constraints is to use the least 1-norm approach, which is useful because synchrotron illuminated tomographic data contains mostly very mild noise, but also has sparsely distributed highly perturbed points for example caused by small dust particles and other detector failures such as those caused by defective scintillator crystals and other mechanical, thermal or electronic causes~\cite{mro14}.

Therefore, our second option for primary optimization problem was given by
\begin{equation}\label{eq:leastn1}
   X_1 = \argmin_{\vect x \in \mathbb R_+^n} \ell( \vect x ) := \| R\vect x - \vect b \|_1.
\end{equation}
Where, now, the primary objective function naturally has an everywhere bounded subdifferential and, as is well known, the $\| \cdot \|_1$ is more forgiving to sparse noise than is the euclidean norm~\cite{dol92,sas86,tbm79,clm73}.

\subsection{Secondary Objective Functions}\label{subsec:secobj}

For the purpose of selecting one among the solutions of~\eqref{eq:leastsq} or~\eqref{eq:leastn1}, we will consider two particular cases of the bilevel program~\eqref{eq:bilevel}.

\subsubsection{Haar 1-Norm}\label{subsubsec:haar}

Here the secondary objective function $f_1$ is given by
\begin{equation*}
   f_{\text{Haar}}( \vect x ) := \| H\vect x \|_1,
\end{equation*}
with $H \in \mathbb R^{n \times n}$ orthonormal, that is, satisfies $H^TH = I$, where $I$ is the identity matrix and superscript $T$ indicates transposition. Transformation $H$ is usually a sparsefying transform, in our case the Haar transform, and one is looking for a sparse (in some cases the sparsest~\cite{crt06,crt06b,don06,don06b,don06c}), in the $H$-transformed space, optimizer of $f_0$ over $X_0$.

Notice that for this particular function we have
\begin{equation}\label{eq:subdiffHaar}
   \partial f_{\text{Haar}}( \vect x ) = H^T\vect\sign( H\vect x ),
\end{equation}
where the set-valued function $\vect\sign : \mathbb R^n \to 2^{\mathbb R^n}$ is given by
\begin{equation*}
   \vect\sign( \vect v ) := \{ \vect x \in \mathbb R^n : x_i \in \sign( v_i ) \}.
\end{equation*}

Let us now consider the soft-thresholding operator, given
componentwise as
\begin{equation*}
   \bigl( \text{\textsc{st}}_\mu( \vect x ) \bigr)_i := 
      \begin{cases}
         x_i + \mu & \text{if}\quad x_i < -\mu\\
         0         & \text{if}\quad x \in [ -\mu, \mu ]\\
         x_i - \mu & \text{if}\quad x_i > \mu.
      \end{cases}
\end{equation*}
We then define an intermediary optimality operator as
\begin{equation}\label{eq:HST}
   \mathcal N_{f_{\text{Haar}}}( \mu, \vect x ) := H^T\text{\textsc{st}}_\mu( H\vect x ).
\end{equation}
A straightforward computation convinces us that
\begin{equation}\label{eq:HSTasSubgrad}
   \mathcal N_{f_{\text{Haar}}}( \mu, \vect x ) = \vect x - \mu \tilde\nabla f_1( \overline{\vect x} ),
\end{equation}
where
\begin{equation*}
   ( H\overline{\vect x} )_i = \begin{cases}
                        ( H\vect x )_i - \sign\bigl( ( H\vect x )_i \bigr)\mu & \text{if}\quad | ( H\vect x )_i | > \mu\\
                        0   & \text{if}\quad | ( H\vect x )_i | \leq \mu
                     \end{cases}.
\end{equation*}
This means that $\overline{\vect x} = \mathcal N_{f_{\text{Haar}}}( \mu, \vect x )$, and we have denoted $\tilde\nabla f_1( \overline{\vect x} ) \in \partial f_{\text{Haar}}( \overline{\vect x} )$.
Now we proceed:
\begin{equation}\label{eq:hstDiffBound}
   \begin{split}
      \| \mathcal N_{f_{\text{Haar}}}( \mu, \vect x ) - \vect y \|^2 &{}= \| \vect x - \mu \tilde\nabla f_{\text{Haar}}( \overline{\vect x} ) - \vect y \|^2\\
      &{} = \| \vect x - \vect y \|^2 - 2\mu\langle \tilde\nabla f_{\text{Haar}}( \overline{\vect x} ), \vect x - \vect y \rangle + \| \mu \tilde\nabla f_{\text{Haar}}( \overline{\vect x} ) \|^2\\
      &{} = \| \vect x - \vect y \|^2 - 2\mu\langle \tilde\nabla f_{\text{Haar}}( \overline{\vect x} ), \overline{\vect x} - \vect y \rangle \\
      &\qquad\qquad\qquad {}- 2\mu\langle \tilde\nabla f_{\text{Haar}}( \overline{\vect x} ), \vect x - \overline{\vect x} \rangle + \| \mu \tilde\nabla f_{\text{Haar}}( \overline{\vect x} ) \|^2
   \end{split}
\end{equation}
and we then can observe from~\eqref{eq:subdiffHaar} that $\partial f_1$ is bounded in a way such that $\| \tilde \nabla f_1( \vect x ) \| \leq \sqrt n$ for every $\vect x$, and from~\eqref{eq:hstDiffBound} that $\| \vect x - \overline{\vect x} \| \leq \mu \sqrt n$ for every $\vect x$. Using this facts followed by the subgradient inequality we have
\begin{equation}\label{eq:propOptHaarSoftThres}
   \begin{split}
      \| \mathcal N_{f_{\text{Haar}}}( \mu, \vect x ) - \vect y \|^2 &{}\leq \| \vect x - \vect y \|^2 - 2\mu\langle \tilde\nabla f_{\text{Haar}}( \overline{\vect x} ), \overline{\vect x} - \vect y \rangle + 3\mu^2n\\
      &{}\leq \| \vect x - \vect y \|^2 - 2\mu\bigl( f_{\text{Haar}}( \overline{\vect x} ) - f_1( \vect y )\bigr) + 3\mu^2n.
   \end{split}
\end{equation}
From where we see that Property~\ref{prop:propopt1} is satisfied with $\rho_1( \mu, k ) = 3\mu n$. Furthermore, since for every $\vect x$ we have $\|\tilde\nabla f_{\text{Haar}}( \vect x )\| \leq \sqrt n$, formulation~\eqref{eq:HSTasSubgrad} of $\mathcal O_{f_1}$ implies that this operator satisfies Property~\ref{prop:propoptbound} with $\gamma = \sqrt n$.

It is worth remarking that the soft-thresholding is an example of a proximal operator,
\begin{equation*}
   \mathcal N_{f_{\text{Haar}}}( \mu, \vect x ) = \argmin_{\vect y \in \mathbb R^n}\left\{ \mu\| H \vect y \|_ 1 + \frac12\| \vect x - \vect y \|^2 \right\}.
\end{equation*}
In fact, it is possible to carry out an argumentation similar to the one that lead to inequality~\eqref{eq:propOptHaarSoftThres} above for more general proximal operators, because they too can be interpreted as examples of implicit subgradient algorithms or, alternatively, as appropriately controlled $\epsilon$-subgradient methods~\cite{col93}.

\subsubsection{Total Variation}\label{subsubsec:tv}

Another function we have used to select among the primary problem optimizers is the Total Variation~\cite{rof92}:
\begin{equation*}
   f_{\text{\textsc{tv}}}( \vect x ) := \sum_{i = 1}^{\sqrt n}\sum_{j = 1}^{\sqrt n}\sqrt{( x_{i, j} - x_{i - 1, j} )^2 + ( x_{i, j} - x_{i, j - 1} )^2},
\end{equation*}
where we have assumed that $n$ is a perfect square, which is true for our reconstructed images, the lexicographical notation:
\begin{equation*}
   x_{i,j} := x_{i + ( j - 1 )\sqrt n}, \quad (i, j) \in \{ 1, 2, \dots, \sqrt{n} \}^2
\end{equation*}
and also the following boundary conditions:
\begin{equation*}
   x_{0, j} = x_{\sqrt{n}, j}, \quad x_{i, 0} = x_{i, \sqrt{n}}, \quad (i, j) \in \{ 1, 2, \dots, \sqrt{n} \}^2.
\end{equation*}

The $f_{\text{\textsc{tv}}}$ subdifferential is bounded and, therefore, application of a subgradient step would satisfy the required operator properties. However given the low computational time of the $f_{\text{\textsc{tv}}}$ subgradient step, it is reasonable to apply it several times instead of a single pass. In this case we can see that the resulting operation would still satisfy the required properties, as follows. Let us denote by $\hat{\mathcal O}_f^J$ the following operator:
\begin{equation}\label{eq:OJ}
   \begin{split}
      \vect x^{( 0 )} & {}:= \vect x\text;\\
      \vect x^{( i )} & {}:= \tilde{\mathcal O}_f( \lambda / i, \vect x^{( i - 1 )} )\text,\quad i \in \{ 1, 2, \dots, J \}\text;\\
      \hat{\mathcal O}_f^J( \lambda, \vect x ) & {}:= \vect x^{( J )}\text.
   \end{split}
\end{equation}
That is, the $J$-fold repetition of $\tilde{\mathcal O}_f$, with diminishing stepsizes. For this kind of operator there holds the following:

\begin{lemma}
   Suppose that $f$ has a bounded subgradient and that $\tilde{\mathcal O}_f$ satisfies Properties~\ref{prop:propopt1}~and~\ref{prop:propoptbound}. Then $\hat{\mathcal O}_f^J$ also satisfies Properties~\ref{prop:propopt1}~and~\ref{prop:propoptbound}.
\end{lemma}
\begin{proof}
   Repeated use of Property~\ref{prop:propoptbound} and the triangle inequality leads to
   \begin{equation*}
      \| \hat{\mathcal O}_f^J( \lambda, \vect x ) - \vect x \| \leq \lambda\gamma\sum_{i = 1}^J\frac 1i\text.
   \end{equation*}
   More generally, telescoping from $j$ to $l$ we have
   \begin{equation}\label{eq:telescopingPropOpBound}
      \| \vect x^{( j )} - \vect x^{( l )} \| \leq \lambda\gamma\sum_{i = j + 1}^l\frac 1i\text.
   \end{equation}
   Now, repeated use of Property~\ref{prop:propopt1} gives
   \begin{equation*}
      \begin{split}
         \| \hat{\mathcal O}_f^J( \lambda, \vect x ) - \vect y \|^2 & {}\leq \| \vect x - \vect y \|^2 - 2\lambda\sum_{i = 1}^J\frac 1i\bigl( f( \vect x^{( i )} ) - f( \vect y ) \bigr) + \lambda\sum_{i = 1}^J\frac 1i\rho( \lambda / i, k )\\
            &{} = \| \vect x - \vect y \|^2 - 2\lambda\sum_{i = 1}^J\frac 1i\bigl( f( \vect x^{( J )} ) - f( \vect y ) \bigr) + \lambda\sum_{i = 1}^J\frac 1i\rho( \lambda / i, k )\\
            &\quad\qquad\qquad\qquad{} + 2\lambda\sum_{i = 1}^J\frac 1i\bigl( f( \vect x^{( J )} ) - f( \vect x^{( i )} ) \bigr).
      \end{split}
   \end{equation*}
   Then, by subgradient boundedness and from~\eqref{eq:telescopingPropOpBound}, we have:
   \begin{equation*}
      \begin{split}
         \| \hat{\mathcal O}_f^J( \lambda, \vect x ) - \vect y \|^2 & {} \leq \| \vect x - \vect y \|^2 - 2\lambda\sum_{i = 1}^J\frac 1i\bigl( f( \vect x^{( J )} ) - f( \vect y ) \bigr) + \lambda\sum_{i = 1}^J\frac 1i\rho( \lambda / i, k )\\
            &\quad\qquad\qquad\qquad{} + 2\lambda M\sum_{i = 1}^J\frac 1i\| \vect x^{( J )} - \vect x^{( i )} \|\\
            & {} \leq \| \vect x - \vect y \|^2 - 2\lambda\sum_{i = 1}^J\frac 1i\bigl( f( \vect x^{( J )} ) - f( \vect y ) \bigr) + \lambda\sum_{i = 1}^J\frac 1i\rho( \lambda / i, k )\\
            &\quad\qquad\qquad\qquad{} + 2\lambda^2\gamma M\sum_{i = 1}^J\frac 1i\sum_{j = i + 1}^J\frac 1j.
      \end{split}
   \end{equation*}
   Therefore, $\hat{\mathcal O}_f^J$ also satisfies Property~\ref{prop:propoptbound}, with error term given by
   \begin{equation*}
      \sum_{i = 1}^J\frac 1i\rho( \lambda / i, k ) + 2\lambda\gamma M\sum_{i = 1}^J\frac 1i\sum_{j = i + 1}^J\frac 1j.\qed
   \end{equation*}
\end{proof}

Notice that if the original error term $\rho( \lambda, k )$ for $\tilde{\mathcal O}_f$ is $O( \lambda )$, then so is the one of the iterated operator $\hat{\mathcal O}_f^J$. This remark will be useful when considering convergence of the incremental algorithm to be presented in Subsection~\ref{subsec:algoinc}.

\section{Numerical Experimentation}

\subsection{Smooth Primary Objective Function}

In the present Subsection, we report the results obtained using Algorithm~\ref{algo:fiba} in order to approximately solve the optimization problem
\begin{equation*}
   \begin{split}
      \min & \quad f_1( \vect x )\\
      \st  & \quad \vect x\in \argmin_{\vect y \in \mathbb R_+^n} \| R\vect y - \vect b \|^2\text,
   \end{split}
\end{equation*}
with both $f_1( \vect x ) = \| H \vect x \|_1$ and $f_1( \vect x ) = f_{\text{\textsc{tv}}}( \vect x )$. Below we describe the parameter selection and explain the rationale supporting each choice.

The bilevel algorithms were compared against the Fast Iterative Soft-Thresholding Algorithm (\textsc{fista})~\cite{bet09} for the function
\begin{equation*}
   \| R\vect x - \vect b \|^2 + \iota_{\mathbb R_+^n}( \vect x ),
\end{equation*}
where $\iota$ is the indicator function~\eqref{eq:iota}. We have used the starting image described in the next paragraph. We iterated \textsc{fista} with a constant stepsize $\lambda = 3$, for the reasons explained in the paragraph following the next.

\paragraph{Starting image} The initial guess $\vect x_0$ used in the experiments presented in this subsection was the zero image.

\paragraph{Stepsize sequences} The sequence $\{\lambda_k\}$ was set to be
\begin{equation}\label{eq:lambdaDiff}
   \lambda_k = \frac{\lambda}{(k + 1)^{0.1}},
\end{equation}
where $\lambda$ was chosen to be just small enough so that the squared residual of the first iterations of the one-level algorithm is decreasing. In the examples, $\lambda = 3$ worked well. This choice for $\lambda$ was noticed to provide the fastest primary objective function decrease during the iterations when using stepsize sequences of the form \eqref{eq:lambdaDiff} above. Sequence $\{ \eta_k \}$ was given by
\begin{equation*}
   \eta_k = \frac{10^6}{( k + 1 )^{0.1}},
\end{equation*}
which was chosen to attain very high values in order not to negatively influence convergence speed by limiting the values of $\lambda_{i_k}$ or $\xi_k$ (both of which remained constant throughout the iterations of Algorithm~\ref{algo:fiba}, without detected need for decreasing). Finally, the stepsize sequence $\{ \mu_k \}$ was given by
\begin{equation*}
   \mu_k = \frac{\mu}{k + 1}.
\end{equation*}
Now, the initial stepsize $\mu$ has been selected in order to make the first pair of subiterations to have a specific relative strength. The precise procedure was the following: we first computed $\vect x_{1 / 3}$ as usual, following Algorithm~\ref{algo:fiba} and using the already chosen $\lambda_0$. Then, a tentative subiteration $\tilde{\vect x}_{2 / 3}$ is computed using a tentative stepsize $\tilde{\mu}_0 = 1$. Finally, the value to be used as starting stepsize is computed by
\begin{equation}\label{eq:mu0Smooth}
   \mu = 10^{-2} \frac{\|\vect x_{0} - \vect x_{1 / 3}\|}{\|\vect x_{1 / 3} - \tilde{\vect x}_{2 / 3}\|}.
\end{equation}
So, the step given by the first subiteration for the primary problem is about $10^2$ times the step given by the subiteration for the secondary optimization problem. This value provided a good compromise between primary problem convergence and secondary function value during the iterations.

\paragraph{Secondary Operators} When the problem being solved had $f_1 = f_{\text{Haar}}$, we used \textsc{fiba} with $\mathcal O_{f_1} = \mathcal N_{f_{\text{Haar}}}$ according to~\eqref{eq:HST}. If the problem had $f_1 = f_{\text{\textsc{tv}}}$, then we have used $\mathcal O_{f_1} = \hat{\mathcal O}_{f_{\text{\textsc{tv}}}}^{10}$ as defined in~\eqref{eq:OJ} with $\tilde{\mathcal O}_{f_{\text{\textsc{tv}}}}( \lambda, \vect x ) = \vect x - \lambda \tilde\nabla f_{\text{\textsc{tv}}}( \vect x )$ where $\tilde\nabla f_{\text{\textsc{tv}}}( \vect x ) \in \partial f_{\text{\textsc{tv}}}( \vect x )$.

\paragraph{Algorithm Convergence} Because all entries of the matrix $R$ are nonnegative and because every pixel of the reconstructed image is crossed by several rays during acquisition, the model has the property that $X_1$ is bounded. Therefore, given the subgradient boundedness of all involved objective functions and the fact that both secondary objective functions $f_{\text{Haar}}$ and $f_{\text{\textsc{tv}}}$ are bounded from below, Theorem~\ref{theo:convFIBA} could be applied to show that Algorithm~\ref{algo:fiba} converges in the cases covered in the present subsection, except for two issues. The first problem is that $\lambda_0$ may not be smaller than $1/L_0$, but we have observed during algorithm evaluation that the alternative sufficient decrease criterion~\eqref{eq:upperBoundLips} was satisfied in every iteration, which suffices, instead, for convergence. Furthermore, boundedness of the iterates was also observed, ensuring convergent behavior without forceful truncation.

\paragraph{Numerical Results} In the plots that follow, we refer to \textsc{fiba} when applied to the model with $f_1 = f_{\text{Haar}}$ as \textsc{fiba-h} and when applied to the model with $f_1 = f_{\text{\textsc{tv}}}$ as \textsc{fiba-tv}. Figure~\ref{fig:phasePlaneDiffProb} shows that, as expected, the bilevel approach influences the iterates of the method, thereby resulting in lower secondary objective function value throughout the iterations. Figure~\ref{fig:iterEvolutionDiffProb}, on the other hand and also unsurprisingly, shows that the convergence speed for the primary optimization problem is reduced as the secondary optimization step influences the iterations. Figure~\ref{fig:eggImages} displays some resulting images, all with the same prescribed primary objective function value, from the algorithms and Figure~\ref{fig:eggProfiles} brings a plot of the profiles across the same line of each image.

We can see the benefits of the enforced smoothing, in particular the good quality of the Total Variation image. The advantage of the bilevel technique here is that the regularization level is chosen by a stopping criterion instead of a parameter in an optimization problem. This latter situation implies that more computational effort would be required since reconstruction for each tentative parameter value would require the equivalent to a large fraction of the computation effort of performing some iterations of the bilevel techniques. We further would like to remark that although heuristic, the accelerating scheme does substantially speed up the method. While we do not show these results in the present paper, when using a non-accelerated technique (i.e, $\eta_k = 0$), the convergence speed is reduced to a rate similar to the \textsc{ista}~\cite{bet09} method.

\begin{figure}
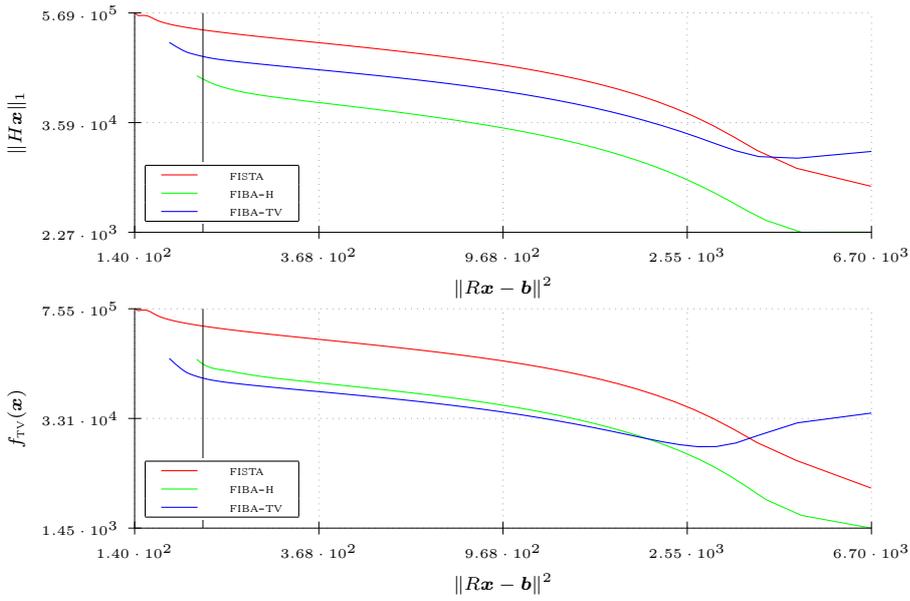

   \input{egg_rec_axisdata.tex}%
   \centering%
   \setlength{\grftotalwidth}{\textwidth}%
   \begin{grfgraphic}[0.3]{%
      \grfxaxis{\fAxfBxaxisValues}{\fAxfBxaxisLabels}%
      \grfyaxis{\fAxfByaxisValues}{\fAxfByaxisLabels}%
      \grfxlabel{\scriptsize$\| R\vect x -  \vect b \|^2$}%
      \grfylabel[L]{\scriptsize$\| H\vect x \|_1$}%
   }%
      \grfxgrid[grfaxisstyle,gray,dotted]{\fAxfBxaxisValues}{\fAxfBxaxisValues}%
      \grfygrid[grfaxisstyle,gray,dotted]{\fAxfByaxisValues}{\fAxfByaxisValues}%
      \draw[red] plot file {egg_fistap_f0xf1.tpl};%
      \draw[green] plot file {egg_bilevel_f0xf1.tpl};%
      \draw[blue] plot file {egg_bilevel_tv_f0xf1.tpl};%
      \draw[grfaxisstyle] (\targetA,0.0) -- (\targetA,1.0);
      {%
         \tiny%
         \grftablegend[bl,fill=white]{|cl|}{%
               \hline%
               \grflegsymbol[2em]{\grflegline[red]} & \textsc{fista}\\
               \grflegsymbol[2em]{\grflegline[green]} & \textsc{fiba-h}\\
               \grflegsymbol[2em]{\grflegline[blue]} & \textsc{fiba-tv}\\
               \hline%
         }%
      }%
   \end{grfgraphic}\\%
   \input{egg_rec_tv_axisdata.tex}%
   \begin{grfgraphic}[0.3]{%
      \grfxaxis{\fAxTVxaxisValues}{\fAxTVxaxisLabels}%
      \grfyaxis{\fAxTVyaxisValues}{\fAxTVyaxisLabels}%
      \grfxlabel{\scriptsize$\| R\vect x -  \vect b \|^2$}%
      \grfylabel[L]{\scriptsize$f_{\text{\textsc{tv}}}( \vect x )$}%
   }%
      \grfxgrid[grfaxisstyle,gray,dotted]{\fAxTVxaxisValues}{\fAxTVxaxisValues}%
      \grfygrid[grfaxisstyle,gray,dotted]{\fAxTVyaxisValues}{\fAxTVyaxisValues}%
      \draw[red] plot file {egg_fistap_f0xTV.tpl};%
      \draw[green] plot file {egg_bilevel_haar_f0xTV.tpl};%
      \draw[blue] plot file {egg_bilevel_f0xTV.tpl};%
      \draw[grfaxisstyle] (\targetA,0.0) -- (\targetA,1.0);
      {%
         \tiny%
         \grftablegend[bl,fill=white]{|cl|}{%
               \hline%
               \grflegsymbol[2em]{\grflegline[red]} & \textsc{fista}\\
               \grflegsymbol[2em]{\grflegline[green]} & \textsc{fiba-h}\\
               \grflegsymbol[2em]{\grflegline[blue]} & \textsc{fiba-tv}\\
               \hline%
         }%
      }%
   \end{grfgraphic}%
   \caption{Logarithmic-scale plots of the trajectories followed by the three studied algorithms over the ``phase-plane'' described by $f_0 \times f_1$. The solid vertical line depicts the residual value of the images shown at Figure~\ref{fig:eggImages}. In both graphics, horizontal scale is $\| R\vect x - \vect b \|^2$. On top, vertical axis depicts $\| H\vect x \|_1$. On bottom, curve height is proportional to $ f_{\text{\textsc{tv}}}( \vect x )$.}%
   \label{fig:phasePlaneDiffProb}%
\end{figure}

\begin{figure}
   \input{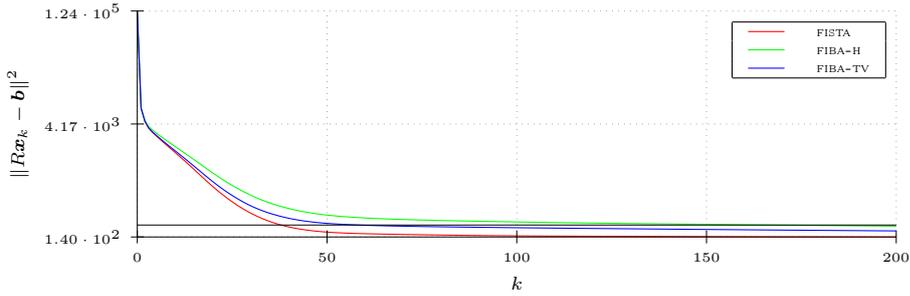}%
   \centering%
   \setlength{\grftotalwidth}{\textwidth}%
   \begin{grfgraphic}[0.3]{%
      \grfxaxis{\iterxfAxaxisValues}{\iterxfAxaxisLabels}%
      \grfyaxis{\iterxfAyaxisValues}{\iterxfAyaxisLabels}%
      \grfylabel[L]{\scriptsize$\| R\vect x_k -  \vect b \|^2$}%
      \grfxlabel{\scriptsize$k$}%
   }%
      \grfxgrid[grfaxisstyle,gray,dotted]{\iterxfAxaxisValues}{\iterxfAxaxisValues}%
      \grfygrid[grfaxisstyle,gray,dotted]{\iterxfAyaxisValues}{\iterxfAyaxisValues}%
      \draw[red] plot file {egg_fistap_iterxf0.tpl};%
      \draw[green] plot file {egg_bilevel_haar_iterxf0.tpl};%
      \draw[blue] plot file {egg_bilevel_tv_iterxf0.tpl};%
      \draw[grfaxisstyle] (0.0,\targetA) -- (1.0,\targetA);
      {%
         \tiny%
         \grftablegend[tr,fill=white]{|cl|}{%
               \hline%
               \grflegsymbol[2em]{\grflegline[red]} & \textsc{fista}\\
               \grflegsymbol[2em]{\grflegline[green]} & \textsc{fiba-h}\\
               \grflegsymbol[2em]{\grflegline[blue]} & \textsc{fiba-tv}\\
               \hline%
         }%
      }%
   \end{grfgraphic}%
   \caption{Evolution of primary objective function over iterations. Solid horizontal line: residual value of the images shown in Figure~\ref{fig:eggImages}.}%
   \label{fig:iterEvolutionDiffProb}%
\end{figure}

\begin{figure}
   \input{egg_rec_axisdata.tex}%
   \centering%
   \setlength{\grftotalwidth}{\textwidth}%
   \begin{grfgraphic}{%
      \def\grfxmin{-1.115}\def\grfxmax{2.0}%
      \def\grfymin{-1.5}\def\grfymax{0.605}%
   }%
      \node[anchor=south,rotate=90,inner sep=\grflabelsep] at (-1.0,0.0) {\scriptsize$2048\times2048$};%
      \node[anchor=south,rotate=90,inner sep=\grflabelsep] at (-1.0,-1.0) {\scriptsize$512\times512$};%
      \node[anchor=south,inner sep=\grflabelsep] at (-0.5,0.5) {\scriptsize\textsc{fista}};%
      \node[anchor=south,inner sep=\grflabelsep] at (0.5,0.5) {\scriptsize\textsc{fiba-h}};%
      \node[anchor=south,inner sep=\grflabelsep] at (1.5,0.5) {\scriptsize\textsc{fiba-tv}};%
      \node[inner sep=0pt] at (-0.5,0.0) {\includegraphics[width=\grfxunit]{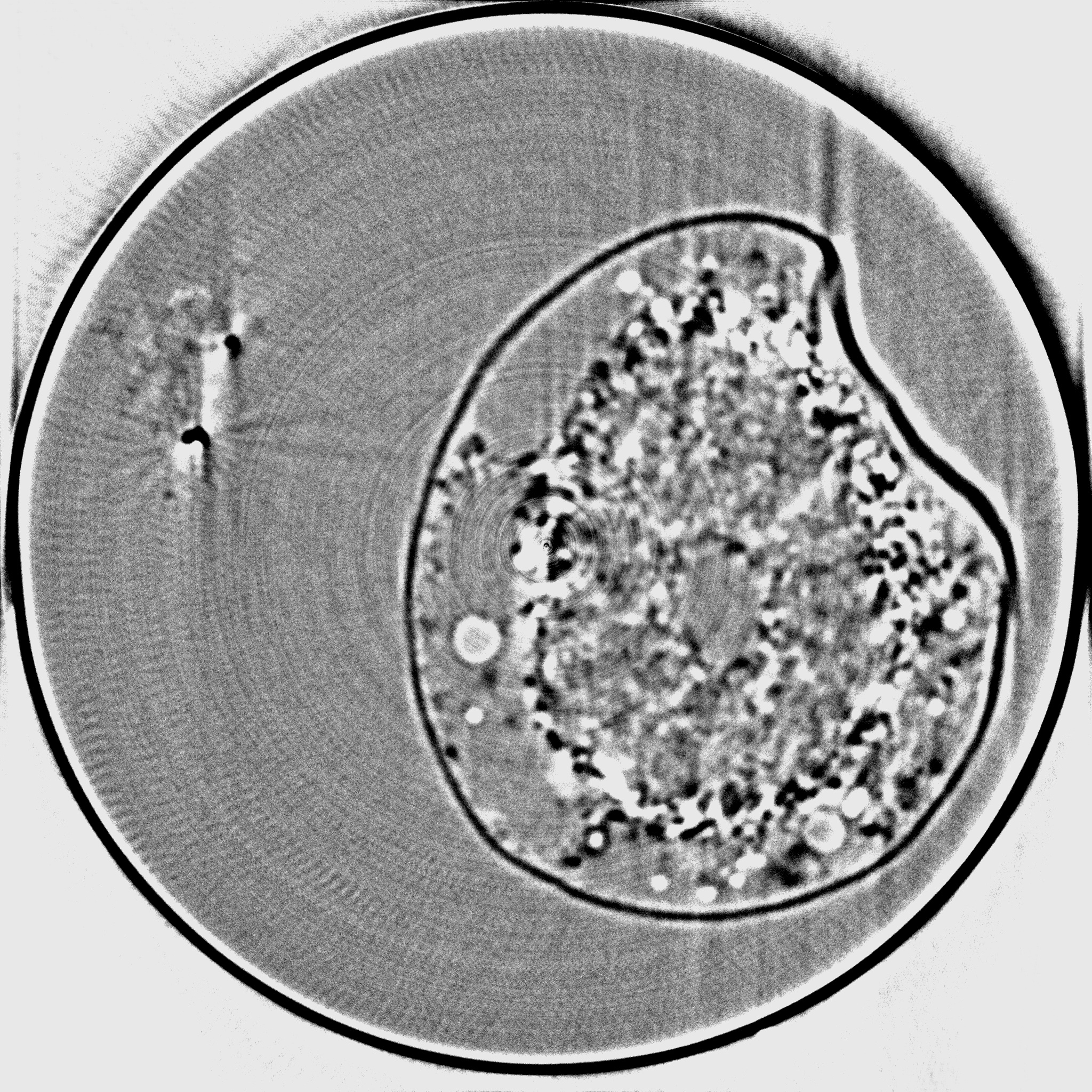}};%
      \node[inner sep=0pt] at ( 0.5,0.0) {\includegraphics[width=\grfxunit]{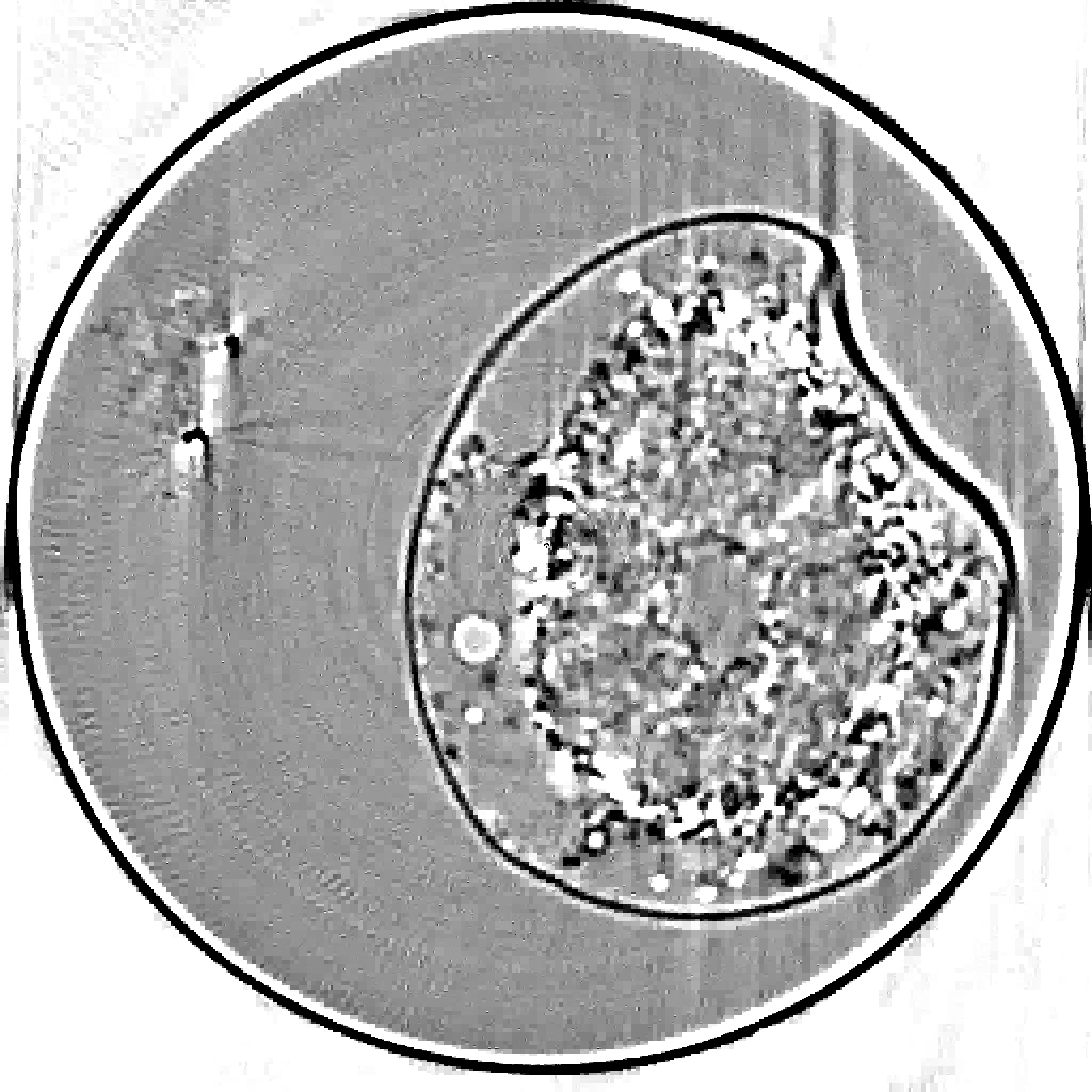}};%
      \node[inner sep=0pt] at ( 1.5,0.0) {\includegraphics[width=\grfxunit]{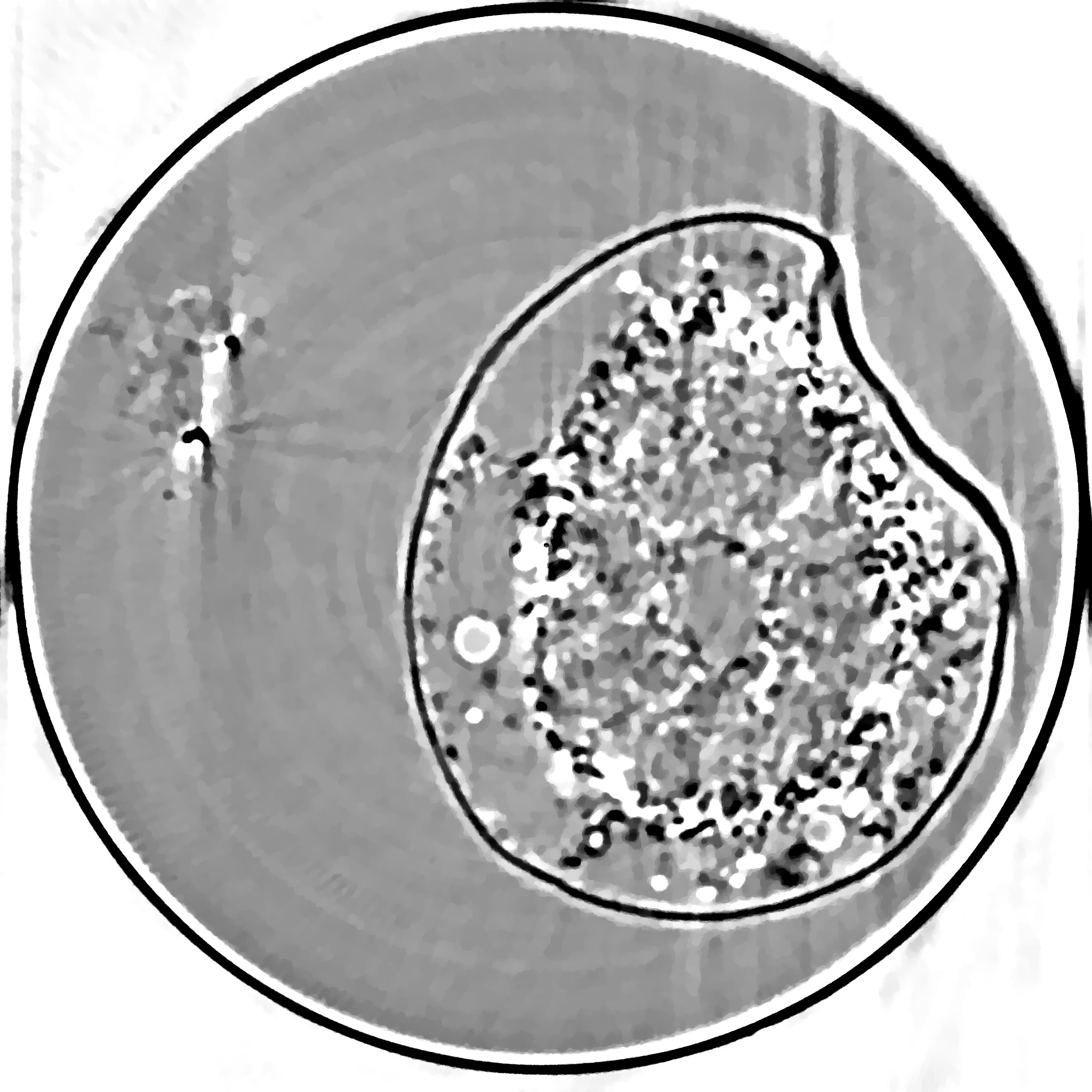}};%
      \draw[grfaxisstyle,red] (-1.0,\eggHprofilePos) -- (0.0,\eggHprofilePos);%
      \draw[grfaxisstyle,green] (0.0,\eggHprofilePos) -- (1.0,\eggHprofilePos);%
      \draw[grfaxisstyle,blue] (1.0,\eggHprofilePos) -- (2.0,\eggHprofilePos);%
      \draw[grfaxisstyle] (-1.0,-0.5) rectangle (0.0,0.5);%
      \draw[grfaxisstyle] (1.0,-0.5) rectangle (0.0,0.5);%
      \draw[grfaxisstyle] (1.0,-0.5) rectangle (2.0,0.5);%
      \node[inner sep=0pt] at (-0.5,-1.0) {\includegraphics[width=\grfxunit]{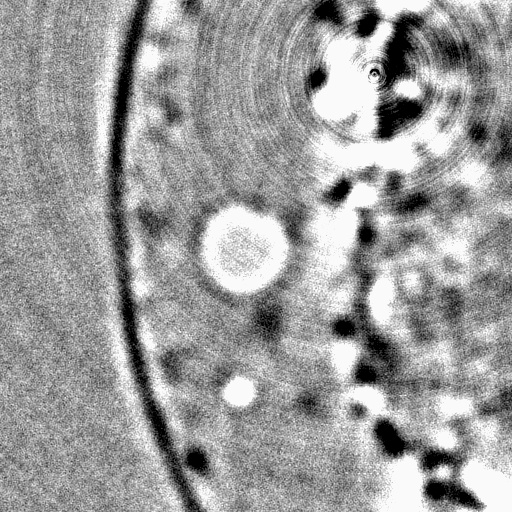}};%
      \node[inner sep=0pt] at ( 0.5,-1.0) {\includegraphics[width=\grfxunit]{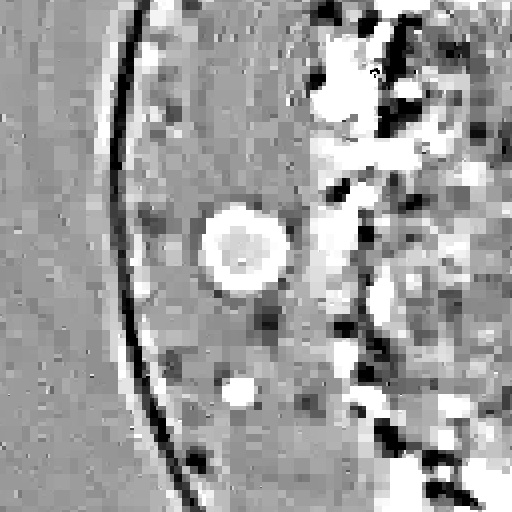}};%
      \node[inner sep=0pt] at ( 1.5,-1.0) {\includegraphics[width=\grfxunit]{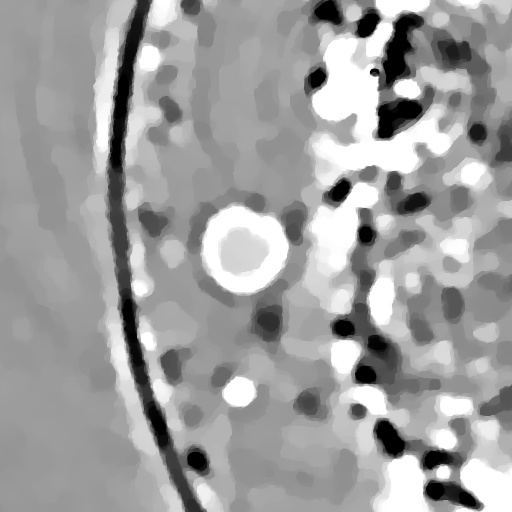}};%
      \draw[grfaxisstyle] (-1.0,-1.5) rectangle (0.0,-0.5);%
      \draw[grfaxisstyle] (1.0,-1.5) rectangle (0.0,-0.5);%
      \draw[grfaxisstyle] (1.0,-1.5) rectangle (2.0,-0.5);%
      \draw[grfaxisstyle] (\eggInsetTLX,\eggInsetTLY) rectangle (\eggInsetBRX,\eggInsetBRY);
      \draw[grfaxisstyle,shift={(1,0)}] (\eggInsetTLX,\eggInsetTLY) rectangle (\eggInsetBRX,\eggInsetBRY);
      \draw[grfaxisstyle,shift={(-1,0)}] (\eggInsetTLX,\eggInsetTLY) rectangle (\eggInsetBRX,\eggInsetBRY);
   \end{grfgraphic}%
   \caption{Reconstructions of fish egg slice from synchrotron radiation transmission data. Top-row: full slice image. Bottom row: details, with location in respective images above shown as a solid black square. The colored solid lines in the top images show the position of the profiles detailed at figure~\ref{fig:eggProfiles}.}\label{fig:eggImages}%
\end{figure}

\begin{figure}
   \input{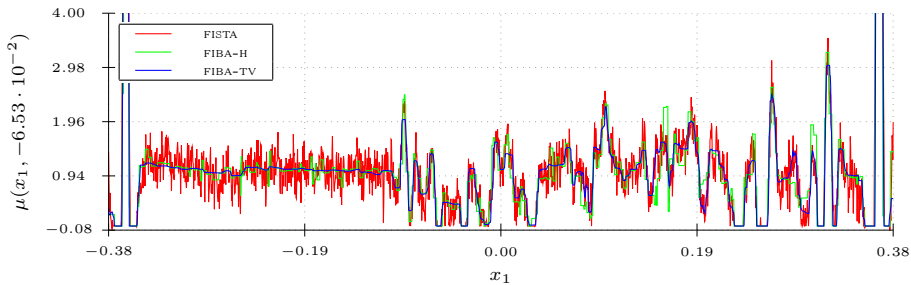}%
   \centering%
   \setlength{\grftotalwidth}{\textwidth}%
   \begin{grfgraphic*}[1.0]{%
      \grfxaxis{\profAxaxisValues}{\profAxaxisLabels}%
      \grfyaxis{\profAyaxisValues}{\profAyaxisLabels}%
      \grfxlabel{\scriptsize$x_1$}%
      \grfylabel[L]{\scriptsize$\mu( x_1, \profApos )$}%
      \def\grfxmin{0.0}\def\grfxmax{1.0}%
   }%
      \grfxgrid[grfaxisstyle,gray,dotted]{\profAxaxisValues}{\profAxaxisValues}%
      \grfygrid[grfaxisstyle,gray,dotted]{\profAyaxisValues}{\profAyaxisValues}%
      \draw[red] plot file {egg_fipa_profileA.tpl};%
      \draw[green] plot file {egg_bilevel_profileA.tpl};%
      \draw[blue] plot file {egg_bilevel_tv_profileA.tpl};%
      {%
         \tiny%
         \grftablegend[tl,fill=white]{|cl|}{%
               \hline%
               \grflegsymbol[2em]{\grflegline[red]} & \textsc{fista}\\
               \grflegsymbol[2em]{\grflegline[green]} & \textsc{fiba-h}\\
               \grflegsymbol[2em]{\grflegline[blue]} & \textsc{fiba-tv}\\
               \hline%
         }%
      }%
   \end{grfgraphic*}%
   \caption{Profiles through the lines indicated in Figure~\ref{fig:eggImages}. It is noticeable the noise suppression characteristics of the Total Variation functional, while retaining image detail.}\label{fig:eggProfiles}%
\end{figure}

\subsection{Non-Differentiable Primary Objective Function}

Now we consider using Algorithm~\ref{algo:iiba} applied to optimization problem
\begin{equation*}
   \begin{split}
      \min & \quad f_1( \vect x )\\
      \st  & \quad \vect x\in \argmin_{\vect y \in \mathbb R_+^n} \| R\vect y - \vect b \|_1\text,
   \end{split}
\end{equation*}
with $f_1( \vect x ) = f_{\text{\textsc{tv}}}( \vect x )$. We do not experiment using $f_1( \vect x ) = \| H \vect x \|_1$ here, we leave it to the simulated experiments presented in the next subsection.

Reconstructions were performed using algorithms where the data was divided in $s$ subsets with $s \in \{1, 2, 4, 8, 16, 32\}$. Each subset was itself composed by data comprising several projection measurements in sequence, i.e., each subset was a vertical stripe of the image\footnote{Coincidently, the vertical colored lines delimit the subsets for the case $s = 4$.} shown in the bottom of Figure~\ref{fig:experimental_Radon}. At each iteration, the sequence of subset processing was selected by a pseudo-random shuffling. This makes the algorithm non-deterministic, but still covered by the theory because every subset was used once in every iteration and each of the corresponding subdifferentials is uniformly bounded. Furthermore, we have observed a consistent behavior among runs and we had not observed a run where sequential data processing led to better convergence than the random ordering.

The various incremental bilevel algorithms were pairwise compared against their pure projected incremental subgradient counterparts, i.e., a variation of Algorithm~\ref{algo:iiba} with $\mu_k \equiv 0$. All of the methods where started with an uniform image as described for the differentiable model. We will denote the bilevel algorithms by \textsc{iiba-}$s$, where $s$ is the number of subsets, and the projected incremental method by \textsc{inc-}$s$.

\paragraph{Starting image} The initial guess $\vect x_0$ used in the experiments of the present section, for all algorithms tested, was a constant image such that $\sum_{i = 1}^m ( R\vect x_0 )_i = \sum_{i = 1}^m b_i$. It is easy to compute the correct constant value $\alpha$ from $\alpha = \sum_{i = 1}^m b_i / \sum_{i = 1}^m ( R\vect 1 )_i$ where $\vect 1$ is the vector, of appropriate dimension, with every coordinate equal to $1$. This choice makes sure that the Radon consistency condition $\sum_{i = 1}^m ( R\vect x )_i = \sum_{i = 1}^m b_i$ is satisfied for the first iteration, potentially avoiding large oscillations in the first steps of the algorithm.

\paragraph{Stepsize sequences} The sequence $\{\lambda_k\}$ for the incremental algorithms with $s$ subsets was set to be
\begin{equation*}
   \lambda_k = \frac{\lambda}{(k + 1)^{\epsilon_s}},
\end{equation*}
where $\lambda$ was of the form
\begin{equation*}
   \lambda = \alpha_s \frac{sf_0( \vect x_0  )}{\|\tilde \nabla f_0( \vect x_0 )\|^2},
\end{equation*}
and $\tilde \nabla f_0( \vect x_0 ) \in \partial f( \vect x_0 )$. The pair $( \alpha_s, \epsilon_s )$ was selected through a simple search procedure as follows. Let us denote by $\vect x_{( s, \alpha, \epsilon )}$ the first iteration to be completed past $4$ seconds of computation time by the pure projected incremental algorithm with $s$ subsets (that is, by \textsc{inc}-$s$) and using parameters $( \alpha, \epsilon )$. Then $( \alpha_s, \epsilon_s )$ was given by
\begin{equation*}
   ( \alpha_s, \epsilon_s ) := \argmin_{ (\alpha, \epsilon ) \in \{ 0.1, 0.2, \dots, 1.0 \} \times \{ 0.5, 0.6, \dots, 0.9 \} } f_0( \vect x_{( s, \alpha, \epsilon )} ).
\end{equation*}
The parameters were only optimized for the non-bilevel case, and the same values were used for the corresponding (with relation to the number of subsets) bilevel algorithm. The second step sequence $\{\mu_k\}$ was prescribed as
\begin{equation*}
   \mu_k = \frac{\mu}{( k + 1 )^{\epsilon_s + 0.1}},
\end{equation*}
with
\begin{equation*}
   \mu = 10^{-1} \frac{\|\vect x_{0} - \vect x_{1 / 3}\|}{\|\vect x_{1 / 3} - \tilde{\vect x}_{2 / 3}\|},
\end{equation*}
where the rationale is the same than in~\eqref{eq:mu0Smooth}, but with target relative importance between the first subiterations of $10$.

\paragraph{Secondary Operators} In the present experiments we have used $\mathcal O_{f_1} = \hat{\mathcal O}_{f_{\text{\textsc{tv}}}}^{5}$ as defined in~\eqref{eq:OJ} with $\tilde{\mathcal O}_{f_{\text{\textsc{tv}}}}( \lambda, \vect x ) = \vect x - \lambda \tilde\nabla f_{\text{\textsc{tv}}}( \vect x )$ where $\tilde\nabla f_{\text{\textsc{tv}}}( \vect x ) \in \partial f_{\text{\textsc{tv}}}( \vect x )$.

\paragraph{Algorithm Convergence} Notice that the stepsizes are of the form
\begin{equation*}
   \lambda_k = \frac{\lambda}{( k + 1 )^\epsilon}\quad\text{and}\quad \mu_k = \frac{\mu}{( k + 1 )^{\epsilon + 0.1}},
\end{equation*}
where $\epsilon \in [0.5, 0.9]$ and $\lambda$ and $\mu$ are nonnegative. It is routine to check that for this range of $\epsilon$ there holds:
\begin{equation*}
   \sum_{k = 0}^\infty \lambda_k = \sum_{k = 0}^\infty \mu_k = \infty, \quad \frac{\mu_k}{\lambda_k} \to 0\quad\text{and}\quad \frac{\lambda_k^2}{\mu_k} \to 0.
\end{equation*}
Also, for the same reasons as in the differentiable primary problem case, the model has the property that $X_1$ is bounded. Given the subgradient boundedness of all involved objective functions and the fact that secondary objective function $f_{\text{\textsc{tv}}}$ is bounded from below, Theorem~\ref{theo:convIIBA} can be applied to prove algorithm convergence.

\begin{figure}
   \input{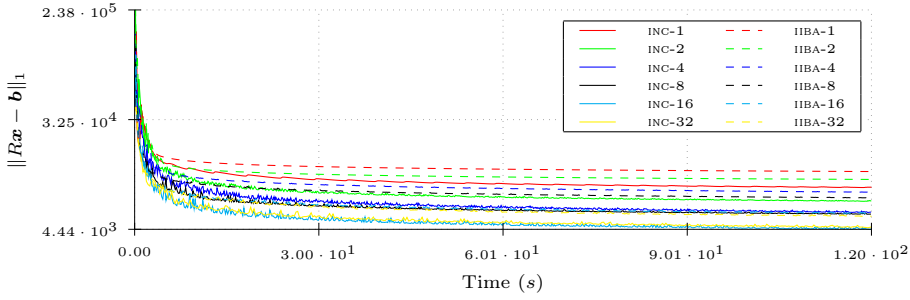}%
   \centering%
   \setlength{\grftotalwidth}{\textwidth}%
   \begin{grfgraphic*}[0.3]{%
      \grfxaxis{\xaxisValues}{\xaxisLabels}%
      \grfyaxis{\yaxisValues}{\yaxisLabels}%
      \grfxlabel{\scriptsize Time ($s$)}%
      \grfylabel[L]{\scriptsize$\| R \vect x - \vect b \|_1$}%
   }%
      \grfxgrid[grfaxisstyle,gray,dotted]{\xaxisValues}{\xaxisValues}%
      \grfygrid[grfaxisstyle,gray,dotted]{\yaxisValues}{\yaxisValues}%
      \draw[red] plot file {inc_times_0.data};%
      \draw[green] plot file {inc_times_2.data};%
      \draw[blue] plot file {inc_times_4.data};%
      \draw[black] plot file {inc_times_6.data};%
      \draw[cyan] plot file {inc_times_8.data};%
      \draw[yellow] plot file {inc_times_10.data};%
      \draw[red,dashed] plot file {inc_times_1.data};%
      \draw[green,dashed] plot file {inc_times_3.data};%
      \draw[blue,dashed] plot file {inc_times_5.data};%
      \draw[black,dashed] plot file {inc_times_7.data};%
      \draw[cyan,dashed] plot file {inc_times_9.data};%
      \draw[yellow,dashed] plot file {inc_times_11.data};%
      {%
         \tiny%
         \grftablegend[tr,fill=white]{|clcl|}{%
               \hline%
               \grflegsymbol[2em]{\grflegline[red]} & \textsc{inc}-1 & \grflegsymbol[2em]{\grflegline[red,dashed]} & \textsc{iiba}-1\\
               \grflegsymbol[2em]{\grflegline[green]} & \textsc{inc}-2 & \grflegsymbol[2em]{\grflegline[green,dashed]} & \textsc{iiba}-2\\
               \grflegsymbol[2em]{\grflegline[blue]} & \textsc{inc}-4 & \grflegsymbol[2em]{\grflegline[blue,dashed]} & \textsc{iiba}-4\\
               \grflegsymbol[2em]{\grflegline[black]} & \textsc{inc}-8 & \grflegsymbol[2em]{\grflegline[black,dashed]} & \textsc{iiba}-8\\
               \grflegsymbol[2em]{\grflegline[cyan]} & \textsc{inc}-16 & \grflegsymbol[2em]{\grflegline[cyan,dashed]} & \textsc{iiba}-16\\
               \grflegsymbol[2em]{\grflegline[yellow]} & \textsc{inc}-32 & \grflegsymbol[2em]{\grflegline[yellow,dashed]} & \textsc{iiba}-32\\
               \hline%
         }%
      }%
   \end{grfgraphic*}
   \caption{Convergence of incremental and incremental bilevel algorithms: primary objective function value as a function of computation time.}\label{fig:nonDiffIncrConvTime}%
\end{figure}

\begin{figure}
   \input{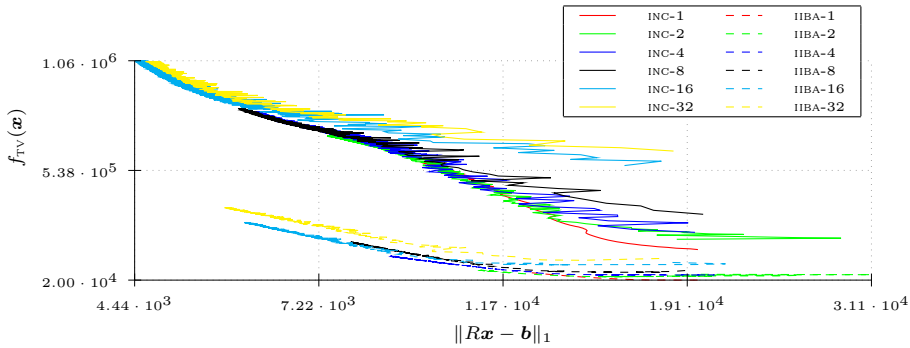}%
   \centering%
   \setlength{\grftotalwidth}{\textwidth}%
   \begin{grfgraphic*}[0.3]{%
      \grfxaxis{\xaxisValues}{\xaxisLabels}%
      \grfyaxis{\yaxisValues}{\yaxisLabels}%
      \grfxlabel{\scriptsize $\| R \vect x - \vect b \|_1$}%
      \grfylabel[L]{\scriptsize $f_{\text{\textsc{tv}}}( \vect x )$}%
      \def\grfymax{1.3}%
   }%
      {%
         \def\grfymax{1}%
         \grfxgrid[grfaxisstyle,gray,dotted]{\xaxisValues}{\xaxisValues}%
         \grfygrid[grfaxisstyle,gray,dotted]{\yaxisValues}{\yaxisValues}%
      }%
      \draw[red] plot file {inc_tv_res_0.data};%
      \draw[green] plot file {inc_tv_res_2.data};%
      \draw[blue] plot file {inc_tv_res_4.data};%
      \draw[black] plot file {inc_tv_res_6.data};%
      \draw[cyan] plot file {inc_tv_res_8.data};%
      \draw[yellow] plot file {inc_tv_res_10.data};%
      \draw[red,dashed] plot file {inc_tv_res_1.data};%
      \draw[green,dashed] plot file {inc_tv_res_3.data};%
      \draw[blue,dashed] plot file {inc_tv_res_5.data};%
      \draw[black,dashed] plot file {inc_tv_res_7.data};%
      \draw[cyan,dashed] plot file {inc_tv_res_9.data};%
      \draw[yellow,dashed] plot file {inc_tv_res_11.data};%
      {%
         \tiny%
         \grftablegend[tr,fill=white]{|clcl|}{%
               \hline%
               \grflegsymbol[2em]{\grflegline[red]} & \textsc{inc}-1 & \grflegsymbol[2em]{\grflegline[red,dashed]} & \textsc{iiba}-1\\
               \grflegsymbol[2em]{\grflegline[green]} & \textsc{inc}-2 & \grflegsymbol[2em]{\grflegline[green,dashed]} & \textsc{iiba}-2\\
               \grflegsymbol[2em]{\grflegline[blue]} & \textsc{inc}-4 & \grflegsymbol[2em]{\grflegline[blue,dashed]} & \textsc{iiba}-4\\
               \grflegsymbol[2em]{\grflegline[black]} & \textsc{inc}-8 & \grflegsymbol[2em]{\grflegline[black,dashed]} & \textsc{iiba}-8\\
               \grflegsymbol[2em]{\grflegline[cyan]} & \textsc{inc}-16 & \grflegsymbol[2em]{\grflegline[cyan,dashed]} & \textsc{iiba}-16\\
               \grflegsymbol[2em]{\grflegline[yellow]} & \textsc{inc}-32 & \grflegsymbol[2em]{\grflegline[yellow,dashed]} & \textsc{iiba}-32\\
               \hline%
         }%
      }%
   \end{grfgraphic*}%
   \caption{Total Variation versus residual norm-1. Notice that the bilevel algorithms present significantly better $f_{\text{\textsc{tv}}}$ values than those for the original model. On the other hand, it is seen here and in Figure~\ref{fig:nonDiffIncrConvTime} that, for example, \textsc{iiba}-32 is competitive, in terms of $f_0$ reduction, with \textsc{inc}-4 while still maintaing a considerable better $f_1$ value than the latter, for the same $f_0$ value.}\label{fig:nonDiffPhasePlane}%
\end{figure}

\paragraph{Numerical Results} By denoting $R_i$, $i \in \{1, 2, \dots, s\}$, the matrix with the rows corresponding to the $i$-th subset, we notice that the computationally demanding parts of the algorithm are products of the form
\begin{equation*}
   R_i\vect x\quad\text{and}\quad R_ i^T\vect y,
\end{equation*}
because the partial subgradients are given by
\begin{equation*}
   R_ i^T\vect{\sign}(R_i \vect x - \vect b).
\end{equation*}
We were not able to make the \textsc{nfft} library as efficient for such partial matrix-vector products, which imposed a large overhead in the partial iterations. We have instead used a ray-tracing algorithm~\cite{hly99,sid85} implemented to run in \textsc{gpu}s (Graphics Processing Units) under single precision floating point arithmetics.

A special scheme was used so that access to slow \textsc{gpu} memory is minimized by performing the computations in sub-images loaded to/from the \textsc{gpu}'s shared memory (a $64$kb fast L1-cache-like memory) in coalesced reads/writes, and summing up the partial results. In this setting, the sequential computation of the $s$ different partial subgradients takes longer than the computation of the subgradient itself, because there are more memory copy to/from shared memory. Yet another per-iteration overhead of the incremental methods are the multiple subiteration updates. Therefore one iteration of the incremental method with $s$ subsets still takes slightly longer than with $s - 1$ subsets.

Even with mandatory overheads, a careful implementation was able to make the iteration-wise speed up provided by the incremental approach advantageous time-wise, as can be seen in Figure~\ref{fig:nonDiffIncrConvTime}. An important feature in this plot is that this speed up is retained by the bilevel algorithms in a similar fashion to the non-bilevel incremental method. That is, if we take into account that considering the secondary objective function in the optimization does, expectedly, slow down the non-incremental method from the viewpoint of primary objective function decrease in comparison to the corresponding non-bilevel algorithm, it can be seen that incremental bilevel techniques too present a speed up in this convergence rate as the number of subsets grow. Computations were performed on a \textsc{gtx} 480 \textsc{gpu} and timing figures were obtained considering iterate updates only, disregarding both data input/output and objective value computation.

A particular point in the experimental results can be seen in Figure~\ref{fig:nonDiffPhasePlane}. In the application of tomographic reconstruction, incrementalism seems to induce more roughness and we therefore see that \textsc{inc}-$s$ achieves lower Total Variation for the same value of 1-norm of the residual than \textsc{inc}-$2s$. Looking at the \textsc{iiba}-$s$ curves in the same plot, we notice that the choice of secondary objective function, which as we have seen is conflicting with the incrementality idea in the one level case, also plays a similar role in the bilevel case. Consequently, at least for our algorithmic parameters selection, there is a incrementality level (or equivalently, primary objective function decrease speed) versus secondary objective function decrease trade-off. Even so, \textsc{iiba}-$32$ provides substantially lower Total Variation for a given 1-norm of the residual compared to \textsc{inc}-$s$ for every $s$ while still achieving faster experimental primary objective function decrease rate than \textsc{inc}-$s$ with $s \leq 4$.

For other bilevel models, at first glance, there seems not to be any reason for an increase in incrementality to lead to worse secondary to primary objective function ratios. However, such antagonism between fast algorithms (with relation to data adherence) and desirable solution properties may appear naturally in models for ill-posed inverse problems like the one we consider here, because in this case overly fit solutions to noisy data are instable and the secondary objective function is usually an attempt at instability prevention.

\subsection{Simulated Data}

The present set of experiments intends to highlight the practical differences and advantages of the bilevel approach
\begin{equation}\label{eq:simbilevel}
   \begin{split}
      \min \quad & \| H\vect x \|_1\\
      \st \quad & \vect x \in\argmin_{\vect y \in \mathbb R^n} \| R\vect y - \vect b \|^2
   \end{split}
\end{equation}
over regularized techniques of the form
\begin{equation}\label{eq:simclassic}
      \min \quad \frac12\| R\vect y - \vect b \|^2 + \gamma \| H\vect x \|_1.
\end{equation}

In order to be able to quantify reconstructed image quality, we use simulated data in these experiments. The ideal image will be denoted by $\vect x^\dagger$ and was a $512 \times 512$ pixels discretized and scaled version of the Shepp-Logan phantom that can be seen at the left of Figure~\ref{fig:rad}. The scaling was such that the relative error after Poisson data simulation was around $10\%$. Tomographic data was computed at $64$ angular samples evenly spaced in $[ 0, \pi )$, each by its turn sampled in $512$ points in $[ -1, 1 ]$.

We have used \textsc{fiba} for solving~\eqref{eq:simbilevel} with just the same parameters (including the secondary operator $\mathcal O_{f_1} = \mathcal N_{f_{\text{Haar}}}$) of those used in Subsection~\ref{subsec:algolip} except that convergence required $\lambda = 2$ and a reasonable starting point for the secondary stepsize sequence was $\mu = 10^2$. Problem~\eqref{eq:simclassic} was solved by the \textsc{fista} algorithm with a constant stepsize $\lambda = 2$. We have tried $\gamma \in \{ 10^2, 10, 1.5, 1, 0\}$. We have run both algorithms for $400$ iterations in this experiment and, as in Subsection~\ref{subsec:algolip}, the algorithms were started with the zero image.

Figure~\ref{fig:slGraphs} shows how the image quality, as measured by the relative error $\| \vect x_k - \vect x^\dagger \| / \| \vect x^\dagger \|$, evolves over the iterations for the tested methods. In Figure~\ref{fig:slImages} we see the best image obtained by each of the methods throughout the iterations. We notice that the \textsc{bilevel} image is competitive with the \textsc{fista} for a certain range of values of $\gamma$, whereas if $\gamma$ is not within the reasonable range, reconstruction by solving~\eqref{eq:simclassic} quickly degrades. In fact, had we used larger values of the starting secondary stepsize good images would still be obtained, maybe requiring more iterations, while larger than ideal values for $\gamma$ in the non-bilevel approach produce a wholly unusable sequence of iterates. On the other hand, smallish secondary stepsize sequences in the bilevel method have practically the same effect than using a small $\gamma$ in the non-bilevel approach, except that the notion of ``small'' includes a wider range of values in the bilevel approach because of the diminishing nature of the sequence $\{\mu_k\}$. Therefore, if we consider the starting secondary stepsize $\mu$ a parameter of the bilevel technique, it is considerably easier to choose than the parameter $\gamma$ of the traditional regularization approach. On the other hand, if a good and efficient procedure for selecting $\gamma$ is available and the regularization function is appropriate, solving~\eqref{eq:simclassic} has the potential of delivering better reconstructions.

\begin{figure}
   \centering%
   \setlength{\grftotalwidth}{\textwidth}%
   \begin{grfgraphic}{%
      \def\grfxmin{-1}\def\grfxmax{2.0}%
      \def\grfymin{-1.5}\def\grfymax{0.5}%
   }%
      \node[inner sep=0pt] at (-0.5,0.0) {\includegraphics[width=\grfxunit]{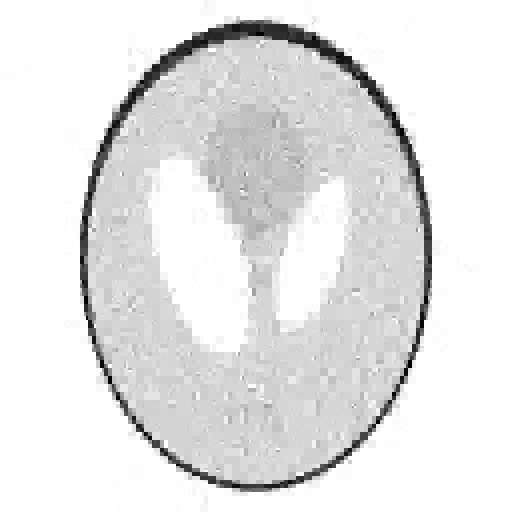}};%
      \input{data_bn}%
      \node[anchor=south west] at (-1.0,-0.5) {\tiny\Tag};%
      \node[anchor=south east] at (0.0,-0.5) {\tiny$\NoiseLevel$};%
      \node[inner sep=0pt] at ( 0.5,0.0) {\includegraphics[width=\grfxunit]{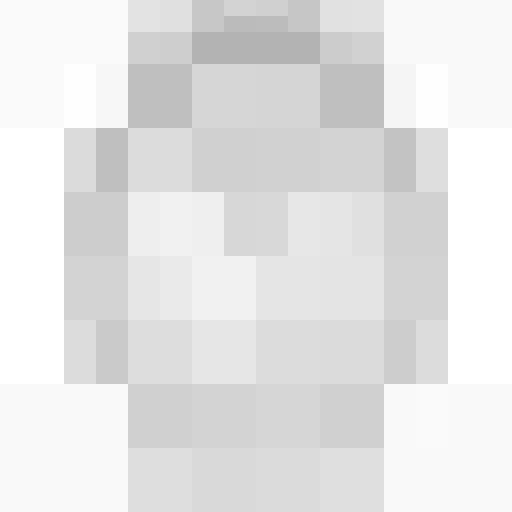}};%
      \input{data_fn_1e2}%
      \node[anchor=south west] at (0.0,-0.5) {\tiny$\Tag$};%
      \node[anchor=south east] at (1.0,-0.5) {\tiny$\NoiseLevel$};%
      \node[inner sep=0pt] at ( 1.5,0.0) {\includegraphics[width=\grfxunit]{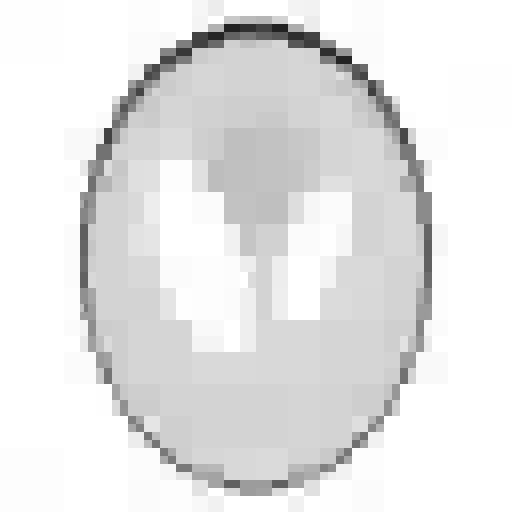}};%
      \input{data_fn_1e1}%
      \node[anchor=south west] at (1.0,-0.5) {\tiny$\Tag$};%
      \node[anchor=south east] at (2.0,-0.5) {\tiny$\NoiseLevel$};%
      \node[inner sep=0pt] at (-0.5,-1.0) {\includegraphics[width=\grfxunit]{{{bi_fn_1.5}}}};%
      \input{data_fn_1.5.tex}%
      \node[anchor=south west] at (-1.0,-1.5) {\tiny$\Tag$};%
      \node[anchor=south east] at (0.0,-1.5) {\tiny$\NoiseLevel$};%
      \node[inner sep=0pt] at ( 0.5,-1.0) {\includegraphics[width=\grfxunit]{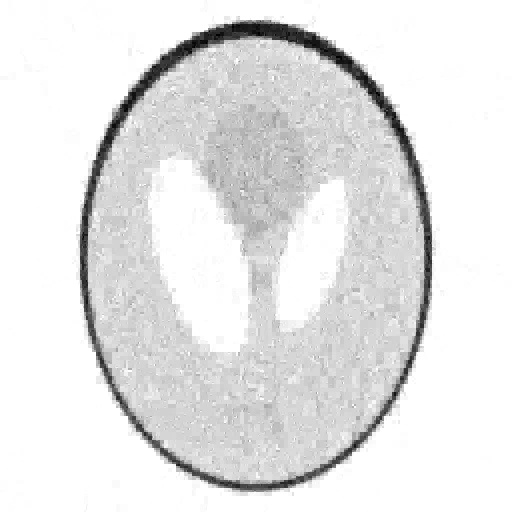}};%
      \input{data_fn_1}%
      \node[anchor=south west] at (0.0,-1.5) {\tiny$\Tag$};%
      \node[anchor=south east] at (1.0,-1.5) {\tiny$\NoiseLevel$};%
      \node[inner sep=0pt] at ( 1.5,-1.0) {\includegraphics[width=\grfxunit]{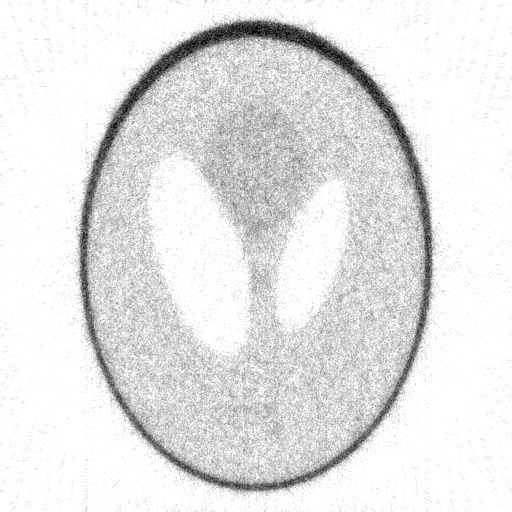}};%
      \input{data_fn_0}%
      \node[anchor=south west] at (1.0,-1.5) {\tiny$\Tag$};%
      \node[anchor=south east] at (2.0,-1.5) {\tiny$\NoiseLevel$};%
      \draw[grfaxisstyle] (-1.0,-0.5) rectangle (0.0,0.5);%
      \draw[grfaxisstyle] (1.0,-0.5) rectangle (0.0,0.5);%
      \draw[grfaxisstyle] (1.0,-0.5) rectangle (2.0,0.5);%
      \draw[grfaxisstyle] (-1.0,-1.5) rectangle (0.0,-0.5);%
      \draw[grfaxisstyle] (1.0,-1.5) rectangle (0.0,-0.5);%
      \draw[grfaxisstyle] (1.0,-1.5) rectangle (2.0,-0.5);%
   \end{grfgraphic}%
   \input{sl_noise_data}%
   \caption{Reconstructions of the Shepp-Logan phantom from simulated noisy data.Top left: best image obtained during execution of \textsc{fiba}. Other images are the best obtained during executions of \textsc{fista}, in these images the bottom-left label is the value of the regularization parameter. The bottom-right label is the relative image error. Relative data error was $\noiselevel$.}\label{fig:slImages}%
\end{figure}

\begin{figure}
   \input{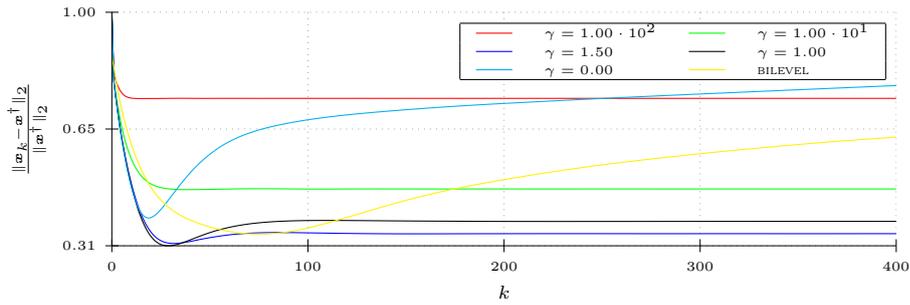}%
   \centering%
   \setlength{\grftotalwidth}{\textwidth}%
   \begin{grfgraphic*}[0.3]{%
      \grfxaxis{\xaxisValues}{\xaxisLabels}%
      \grfyaxis{\yaxisValues}{\yaxisLabels}%
      \grfxlabel{\scriptsize$k$}%
      \grfylabel[L]{\scriptsize$\frac{\| \vect x_k - \vect x^\dagger \|_2}{\| \vect x^\dagger \|_2}$}%
   }%
      \grfxgrid[grfaxisstyle,gray,dotted]{\xaxisValues}{\xaxisValues}%
      \grfygrid[grfaxisstyle,gray,dotted]{\yaxisValues}{\yaxisValues}%
      \draw[red] plot file {sl_noisy_iter_error_0.data};%
      \draw[green] plot file {sl_noisy_iter_error_1.data};%
      \draw[blue] plot file {sl_noisy_iter_error_2.data};%
      \draw[black] plot file {sl_noisy_iter_error_3.data};%
      \draw[cyan] plot file {sl_noisy_iter_error_4.data};%
      \draw[yellow] plot file {sl_noisy_iter_error_5.data};%
      {%
         \tiny%
         \grftablegend[tr,fill=white]{|clcl|}{%
               \hline%
               \grflegsymbol[2em]{\grflegline[red]} & $\gamma = 1.00 \cdot 10^2$ & \grflegsymbol[2em]{\grflegline[green]} & $\gamma = 1.00 \cdot 10^1$\\
               \grflegsymbol[2em]{\grflegline[blue]} & $\gamma = 1.50$ & \grflegsymbol[2em]{\grflegline[black]} & $\gamma = 1.00$\\
               \grflegsymbol[2em]{\grflegline[cyan]} & $\gamma = 0.00$ & \grflegsymbol[2em]{\grflegline[yellow]} & \textsc{bilevel}\\
               \hline%
         }%
      }%
   \end{grfgraphic*}
   \caption{Image quality evolution as iterations proceeds in simulated experiment.}\label{fig:slGraphs}%
\end{figure}

\section{Conclusions}

The present paper introduced an abstract class of explicit numerical methods for instances of bilevel non-differentiable convex optimization problems and proposed two first-order concrete representatives of these new algorithms with reduced per iteration computational cost. The proposed methods have computationally simple iterations. The reported numerical experimentation showed that, when used in tomographic reconstruction problems where the problem size is huge, the low computational complexity of the iterations and the advanced modeling allow to the technique to generate high quality images at a moderate computational cost from sparsely sampled and noisy data. Algorithmic flexibility was highlighted by two conceptually distinct implementations: on one side, an implementation in the high-level Python language, which is free and portable to most computing architectures and environments. On the other hand, the simplicity of the method was also suitable for a low-level hardware-specific implementation fully running on a \textsc{gpu} with no external software library dependency.

\section*{Acknowledgements}

We would like to thank the LNLS for providing the beam time for the tomographic acquisition, obtained under proposal number 17338. We are also grateful to Prof. Marcelo dos Anjos (Federal University of Amazonas) for kindly providing the fish egg samples used in the presented experimentation and Dr. Eduardo X. Miqueles for invaluable help in data acquisition and for the discussions about tomographic reconstruction. We are also indebted to the anonymous referees who gave numerous suggestions that lead to the improvement of the original manuscript.

   \bibliography{paper.bib}

\begin{thebibliography}{38}
\providecommand{\natexlab}[1]{#1}
\providecommand{\url}[1]{\texttt{#1}}
\expandafter\ifx\csname urlstyle\endcsname\relax
  \providecommand{\doi}[1]{doi: #1}\else
  \providecommand{\doi}{doi: \begingroup \urlstyle{rm}\Url}\fi

\bibitem[Beck and Sabach(2014)]{bes14}
{\sc Amir Beck and Shohan Sabach}.
\newblock A first order method for finding minimal norm-like solutions of
  convex optimization problems.
\newblock \emph{Mathematical Programming}, 147\penalty0 (1):\penalty0 25--46,
  2014.
\newblock \doi{10.1007/s10107-013-0708-2}.

\bibitem[Beck and Teboulle(2009)]{bet09}
{\sc Amir Beck and Marc Teboulle}.
\newblock A fast iterative shrinkage-thresholding algorithm for linear inverse
  problems.
\newblock \emph{SIAM Journal on Imaging Sciences}, 2\penalty0 (1):\penalty0
  183--202, 2009.
\newblock \doi{10.1137/080716542}.

\bibitem[Bertsekas(2011)]{ber11}
{\sc Dimitri~P. Bertsekas}.
\newblock Incremental proximal methods for large scale convex optimization.
\newblock \emph{Mathematical Programming}, 129\penalty0 (2):\penalty0 163--195,
  2011.
\newblock \doi{10.1007/s10107-011-0472-0}.

\bibitem[Bertsekas(1997)]{ber97}
{\sc Dimitri~P. Bertsekas}.
\newblock A new class of incremental gradient methods for least squares
  problems.
\newblock \emph{SIAM Journal on Optimization}, 7\penalty0 (4):\penalty0
  913--926, 1997.
\newblock \doi{10.1137/S1052623495287022}.

\bibitem[Bertsekas and Tsitsiklis(2000)]{bet00}
{\sc Dimitri~P. Bertsekas and John~N. Tsitsiklis}.
\newblock Gradient convergence in gradient methods with errors.
\newblock \emph{SIAM Journal on Optimization}, 10\penalty0 (3):\penalty0
  627--642, 2000.
\newblock \doi{10.1137/S1052623497331063}.

\bibitem[Blatt et~al.(2007)Blatt, Hero and Gauchman]{bhg07}
{\sc Doron Blatt, Alfred~O. Hero and Hillel Gauchman}.
\newblock A convergent incremental gradient method with a constant step size.
\newblock \emph{SIAM Journal on Optimization}, 18\penalty0 (1):\penalty0
  29--51, 2007.
\newblock \doi{10.1137/040615961}.

\bibitem[Bredies and Zhariy(2013)]{brz13}
{\sc Kristian Bredies and Mariya Zhariy}.
\newblock A discrepancy-based parameter adaptation and stopping rule for
  minimization algorithms aiming at {T}ikhonov-type regularization.
\newblock \emph{Inverse Problems}, 29\penalty0 (2):\penalty0 025008, 2013.
\newblock \doi{10.1088/0266-5611/29/2/025008}.

\bibitem[Cabot(2005)]{cab05}
{\sc Alexandre Cabot}.
\newblock Proximal point algorithm controlled by a slowly vanishing term:
  Applications to hierarchical minimization.
\newblock \emph{SIAM Journal on Optimization}, 15\penalty0 (2):\penalty0
  555--572, 2005.
\newblock \doi{10.1137/S105262340343467X}.

\bibitem[Cand{\`e}s et~al.(2006{\natexlab{a}})Cand{\`e}s, Romberg and
  Tao]{crt06}
{\sc Emmanuel~J. Cand{\`e}s, Justin~K. Romberg and Terence Tao}.
\newblock Robust uncertainty principles: Exact signal reconstruction from
  highly incomplete frequency information.
\newblock \emph{IEEE Transactions on Information Theory}, 52\penalty0
  (2):\penalty0 489--509, 2006{\natexlab{a}}.
\newblock \doi{10.1109/TIT.2005.862083}.

\bibitem[Cand{\`e}s et~al.(2006{\natexlab{b}})Cand{\`e}s, Romberg and
  Tao]{crt06b}
{\sc Emmanuel~J. Cand{\`e}s, Justin~K. Romberg and Terence Tao}.
\newblock Stable signal recovery from incomplete and inaccurate measurements.
\newblock \emph{Communications on Pure and Applied Mathematics}, 59\penalty0
  (8):\penalty0 1207--1223, 2006{\natexlab{b}}.
\newblock \doi{10.1002/cpa.20124}.

\bibitem[Claerbout and Muir(1973)]{clm73}
{\sc Jon~F. Claerbout and Francis Muir}.
\newblock Robust modelling with erratic data.
\newblock \emph{Geophysics}, 38\penalty0 (5):\penalty0 826--844, 1973.
\newblock \doi{10.1190/1.1440378}.

\bibitem[Correa and Lemar{\'e}chal(1993)]{col93}
{\sc Rafael Correa and Claude Lemar{\'e}chal}.
\newblock Convergence of some algorithms for convex minimization.
\newblock \emph{Mathematical Programming}, 62:\penalty0 261--275, 1993.
\newblock \doi{10.1007/BF01585170}.

\bibitem[{De Pierro} and Yamagishi(2001)]{dey01}
{\sc {\'A}lvaro~Rodolfo {De Pierro} and Michel Eduardo~Beleza Yamagishi}.
\newblock Fast {EM}-like methods for maximum ``a posteriori'' estimates in
  emission tomography.
\newblock \emph{IEEE Transactions on Medical Imaging}, 20\penalty0
  (4):\penalty0 280--288, 2001.
\newblock \doi{10.1109/42.921477}.

\bibitem[Donoho and Logan(1992)]{dol92}
{\sc D.~L. Donoho and B.~F. Logan}.
\newblock Signal recovery and the large sieve.
\newblock \emph{SIAM Journal on Applied Mathematics}, 52\penalty0 (2):\penalty0
  577--591, 1992.
\newblock \doi{10.1137/0152031}.

\bibitem[Donoho(2006{\natexlab{a}})]{don06}
{\sc David~L. Donoho}.
\newblock Compressed sensing.
\newblock \emph{IEEE Transactions on Information Theory}, 52\penalty0
  (4):\penalty0 1289--1306, 2006{\natexlab{a}}.
\newblock \doi{10.1109/TIT.2006.871582}.

\bibitem[Donoho(2006{\natexlab{b}})]{don06b}
{\sc David~L. Donoho}.
\newblock For most large underdetermined systems of linear equations, the
  minimal $\ell_1$-norm solution is also the sparsest solution.
\newblock \emph{Communications on Pure and Applied Mathematics}, 59\penalty0
  (6):\penalty0 797--829, 2006{\natexlab{b}}.
\newblock \doi{10.1002/cpa.20132}.

\bibitem[Donoho(2006{\natexlab{c}})]{don06c}
{\sc David~L. Donoho}.
\newblock For most large underdetermined systems of equations, the minimal
  $\ell_1$-norm near-solution approximates the sparsest near-solution.
\newblock \emph{Communications on Pure and Applied Mathematics}, 59\penalty0
  (7):\penalty0 907--934, 2006{\natexlab{c}}.
\newblock \doi{10.1002/cpa.20131}.

\bibitem[Fourmont(2003)]{fou03}
{\sc Karsten Fourmont}.
\newblock Non-equispaced fast {F}ourier transforms with applications to
  tomography.
\newblock \emph{Journal of Fourier Analysis and Applications}, 9\penalty0
  (5):\penalty0 431--450, 2003.
\newblock \doi{10.1007/s00041-003-0021-1}.

\bibitem[Gardu{\~{n}}o and Herman(2014)]{gah14}
{\sc Edgar Gardu{\~{n}}o and Gabor~T. Herman}.
\newblock Superiorization of the {ML-EM} algorithm.
\newblock \emph{IEEE Transactions on Nuclear Science}, 61\penalty0
  (1):\penalty0 162--172, 2014.
\newblock ISSN 00189499.
\newblock \doi{10.1109/TNS.2013.2283529}.

\bibitem[Han et~al.(1999)Han, Liang and You]{hly99}
{\sc Guoping Han, Zhengrong Liang and Jiangsheng You}.
\newblock A fast ray-tracing technique for {TCT} and {ECT} studies.
\newblock \emph{Conference Records of the 1999 IEEE Nuclear Science Symposium},
  \penalty0 (3):\penalty0 1515--1518, 1999.
\newblock \doi{10.1109/NSSMIC.1999.842846}.

\bibitem[Helou et~al.(2014)Helou, Censor, Chen, Chern, {De Pierro}, Jiang and
  Lu]{hcc14}
{\sc Elias~S. Helou, Yair Censor, Tai-Been Chen, I-Liang Chern, {\'A}lvaro~R.
  {De Pierro}, Ming Jiang and Henry H.-S. Lu}.
\newblock String-averaging expectation-maximization for maximum likelihood
  estimation in emission tomography.
\newblock \emph{Inverse Problems}, 30\penalty0 (5):\penalty0 055003, 2014.
\newblock \doi{10.1088/0266-5611/30/5/055003}.

\bibitem[{Helou Neto} and {De Pierro}(2011)]{hed11}
{\sc Elias~Salom{\~a}o {Helou Neto} and {\'A}lvaro~Rodolfo {De Pierro}}.
\newblock On perturbed steepest descent methods with inexact line search for
  bilevel convex optimization.
\newblock \emph{Optimization}, 60\penalty0 (8-9):\penalty0 991--1008, 2011.
\newblock \doi{10.1080/02331934.2010.536231}.

\bibitem[Herman(1980)]{her80}
{\sc Gabor~T. Herman}.
\newblock \emph{Image Reconstruction from Projections: The Fundamentals of
  Computerized Tomography}.
\newblock Academic Press, 1980.

\bibitem[Huber(1964)]{hub64}
{\sc Peter~J. Huber}.
\newblock Robust estimation of a location parameter.
\newblock \emph{The Annals of Mathematical Statistics}, 35\penalty0
  (1):\penalty0 73--101, 1964.
\newblock URL \url{http://projecteuclid.org/euclid.aoms/1177703732}.

\bibitem[Kak and Slaney(1988)]{kas88}
{\sc Avinash~C. Kak and Malcolm Slaney}.
\newblock \emph{Principles of Computerized Tomographic Imaging}.
\newblock IEEE press, 1988.

\bibitem[Keiner et~al.(2009)Keiner, Kunis and Potts]{kkp09}
{\sc James Keiner, Stefan Kunis and Daniel Potts}.
\newblock Using {NFFT3} -- a software library for various nonequispaced fast
  {Fourier} transforms.
\newblock \emph{ACM Transactions on Mathematical Software}, 2009.
\newblock \doi{10.1145/1555386.1555388}.

\bibitem[Miqueles et~al.(2014)Miqueles, Rinkel, O'Dowd and Berm\'udez]{mro14}
{\sc Eduardo~X. Miqueles, Jean Rinkel, Frank O'Dowd and Juan S.~V. Berm\'udez}.
\newblock Generalized {T}itarenko's algorithm for ring artefacts reduction.
\newblock \emph{Journal of Synchrotron Radiation}, 21:\penalty0 1333--1346,
  2014.
\newblock \doi{10.1107/S1600577514016919}.

\bibitem[Natterer(1986)]{nat86}
{\sc Frank Natterer}.
\newblock \emph{The Mathematics of Computerized Tomography}.
\newblock Wiley, 1986.

\bibitem[Nedi{\'c} and Bertsekas(2001)]{neb01}
{\sc Angelia Nedi{\'c} and Dimitri~P. Bertsekas}.
\newblock Incremental subgradient methods for nondifferentiable optimization.
\newblock \emph{SIAM Journal on Optimization}, 12\penalty0 (1):\penalty0
  109--138, 2001.
\newblock \doi{10.1137/S1052623499362111}.

\bibitem[Rockafellar(1976)]{roc76}
{\sc R.~Tyrrell Rockafellar}.
\newblock Monotone operators and the proximal point algorithm.
\newblock \emph{SIAM Journal on Control and Optimization}, 14\penalty0
  (5):\penalty0 877--898, 1976.
\newblock \doi{10.1137/0314056}.

\bibitem[Rudin et~al.(1992)Rudin, Osher and Fatemi]{rof92}
{\sc Leonid~I. Rudin, Stanley Osher and Emad Fatemi}.
\newblock Nonlinear total variation based noise removal algorithms.
\newblock \emph{Physica D: Nonlinear Phenomena}, 60\penalty0 (1-4):\penalty0
  259--268, 1992.
\newblock \doi{10.1016/0167-2789(92)90242-F}.

\bibitem[Santosa and Symes(1986)]{sas86}
{\sc Fadil Santosa and William~W. Symes}.
\newblock Linear inversion of band-limited reflection seismograms.
\newblock \emph{SIAM Journal on Scientific and Statistical Computing},
  7\penalty0 (4):\penalty0 1307--1330, 1986.
\newblock \doi{10.1137/0907087}.

\bibitem[Siddon(1985)]{sid85}
{\sc Robert~L. Siddon}.
\newblock Fast calculation of the exact radiological path for a
  three-dimensional {CT} array.
\newblock \emph{Medical Physics}, 12\penalty0 (2):\penalty0 252--255, 1985.
\newblock \doi{10.1118/1.595715}.

\bibitem[Solodov(2007)]{sol07}
{\sc Mikhail Solodov}.
\newblock An explicit descent method for bilevel convex optimization.
\newblock \emph{Journal of Convex Analysis}, 14\penalty0 (2):\penalty0
  227--237, 2007.
\newblock URL \url{http://www.heldermann.de/JCA/JCA14/JCA142/jca14016.htm}.

\bibitem[Solodov(2008)]{sol08}
{\sc Mikhail Solodov}.
\newblock A bundle method for a class of bilevel nonsmooth convex minimization
  problems.
\newblock \emph{SIAM Journal on Optimization}, 18\penalty0 (1):\penalty0
  242--259, 2008.
\newblock \doi{10.1137/050647566}.

\bibitem[Solodov(1998)]{sol98}
{\sc Mikhail~V. Solodov}.
\newblock Incremental gradient algorithms with stepsizes bounded away from
  zero.
\newblock \emph{Computational Optimization and Applications}, 11\penalty0
  (1):\penalty0 23--35, 1998.
\newblock \doi{10.1023/A:1018366000512}.

\bibitem[Solodov and Zavriev(1998)]{soz98}
{\sc Mikhail~V. Solodov and S.~K. Zavriev}.
\newblock Error stability properties of generalized gradient-type algorithms.
\newblock \emph{Journal of Optimization Theory and Applications}, 98\penalty0
  (3):\penalty0 663--680, 1998.
\newblock \doi{10.1023/A:1022680114518}.

\bibitem[Taylor et~al.(1979)Taylor, Banks and McCoy]{tbm79}
{\sc Howard~L. Taylor, Stephen~C. Banks and John~F. McCoy}.
\newblock Deconvolution with the $\ell_1$ norm.
\newblock \emph{Geophysics}, 44\penalty0 (1):\penalty0 39--52, 1979.
\newblock \doi{10.1190/1.1440921}.

\end{thebibliography}
   \bibliographystyle{paper}

\end{document}